\documentclass[a4paper,11pt]{article}
\usepackage[latin1]{inputenc}
\usepackage[T1]{fontenc}
\usepackage[english]{babel}
\usepackage{amsmath}
\usepackage{amsfonts}
\usepackage{amssymb}
\usepackage{amsthm}
\usepackage{verbatim}  
\usepackage{url}       
\usepackage{graphicx}
\usepackage{enumerate}  
\usepackage{float}

\newcommand{\psiav}{\psi_{av}}
\newcommand{\psieps}{\psi_{\varepsilon}}
\newcommand{\md}{\mathrm{d}}
\newcommand{\lag}{\langle}
\newcommand{\rag}{\rangle}

\newcommand{\CC}{\mathcal{C}}
\newcommand{\LL}{\mathcal{L}}

\newcommand{\R}{\mathbb{R}}
\newcommand{\C}{\mathbb{C}}
\newcommand{\N}{\mathbb{N}}

\numberwithin{equation}{section}

\theoremstyle{definition}

\newtheorem{hypo}{Hypotheses}[section]

\theoremstyle{plain}
\newtheorem{prop}{Proposition}[section]
\newtheorem{lemme}{Lemma}[section]
\newtheorem{theo}{Theorem}[section]

\theoremstyle{remark}
\newtheorem{rmq}{Remark}[section]

\title{Explicit approximate controllability of the Schr\"odinger equation with a polarizability term}
\author{
Morgan \textsc{Morancey}
\footnote{CMLA UMR 8536, ENS Cachan, 61 avenue du Président Wilson, 94235 Cachan, FRANCE.
email: Morgan.Morancey@cmla.ens-cachan.fr 
\newline \indent
CMLS UMR 7640, Ecole Polytechnique, 91128 Palaiseau, FRANCE.}
\thanks{The author was partially supported by the ``Agence Nationale de la Recherche'' (ANR),
Projet Blanc EMAQS number ANR-2011-BS01-017-01.}
}

\begin{document}

\maketitle


\hrule
\begin{abstract}
We consider a controlled Schr\"odinger equation with a dipolar and a polarizability term, used when the dipolar approximation is not valid. The control is the amplitude of the external electric field, it acts non linearly on the state. We extend in this infinite dimensional framework previous techniques used by Coron, Grigoriu, Lefter and Turinici for stabilization in finite dimension. We consider a highly oscillating control and prove the semi-global weak $H^2$ stabilization of the averaged system using a Lyapunov function introduced by Nersesyan. Then it is proved that the solutions of the Schr\"odinger equation and of the averaged equation stay close on every finite time horizon provided that the control is oscillating enough. Combining these two results, we get approximate controllability to the ground state for the polarizability system with explicit controls. Numerical simulations are presented to illustrate those theoretical results.
\end{abstract}
\hrule

$\\$


\section{Introduction}

\subsection{Main result}
\label{sect_main_result}

$\indent$ Following \cite{RouchonModele} we consider a quantum particle in a potential $V(x)$ and an electric field of amplitude $u(t)$. We assume that the dipolar approximation is not valid (see \cite{Dion_1,Dion_2}). Then, the particle is represented by its wave function $\psi(t,x)$ solution of the following Schr\"odinger equation
\begin{equation}
\label{syst}
\left\{
\begin{aligned}
i\partial_t \psi &= (-\Delta + V(x)) \psi  + u(t)Q_1(x) \psi + u(t)^2 Q_2(x) \psi, \quad  x \in D ,
\\
\psi_{| \partial D} &= 0 ,
\end{aligned}
\right.
\end{equation}
with initial condition
\begin{equation}
\psi(0,x) = \psi^0(x), \quad x \in D,
\label{syst3}
\end{equation}
where $D \subset \R^m$ is a bounded domain with smooth boundary. The functions $V, Q_1, Q_2 \in C^\infty(\overline{D},\mathbb{R})$ are given, $Q_1$ is the dipolar moment and $Q_2$ the polarizability moment.
For the sake of simplicity, we denote by $L^2$, $H^1_0$ and $H^2$ respectively the usual Lebesgue and Sobolev spaces $L^2(D,\C)$, $H^1_0(D,\C)$ and $H^2(D,\C)$.
The following well-posedness result holds (see \cite{Caz}) by application of the Banach fixed point theorem.
\begin{prop}
\label{existence}
For any $\psi^0 \in H^1_0 \cap H^2$ and $u \in L^2_{loc}([0,+\infty), \mathbb{R})$, the system (\ref{syst})-(\ref{syst3}) has a unique weak solution $\psi \in C^0([0,+\infty), H^1_0 \cap H^2)$. Moreover, for all $t>0$, $||\psi(t,.)||_{L^2} = ||\psi^0||_{L^2}$ and there exists $C=C(Q_1,Q_2)>0$ such that for any $t>0$,
\begin{equation*}
|| \psi(t,\cdot) ||_{H^2} \leq || \psi^0 ||_{H^2} e^{C \int_0^t |u(\tau)| + |u(\tau)|^2 \md \tau} .
\end{equation*}
\end{prop}

Let $S:=\big\{ \psi \in L^2(D,\mathbb{C}) ; ||\psi||_{L^2}=1 \big\}$ and $\lag \cdot ,\cdot \rag$ be the usual scalar product on $L^2(D,\mathbb{C})$
\begin{equation*}
\lag f, g \rag = \int_D f(x) \overline{g(x)} \md x, \quad \text{for } f,g \in L^2(D,\mathbb{C}).
\end{equation*}
 We consider the operator $(-\Delta + V)$ with domain $H^1_0 \cap H^2$ and denote by $(\lambda_k)_{k\in \mathbb{N}^*}$ the non decreasing sequence of its eigenvalues and by $(\phi_k)_{k\in \mathbb{N}^*}$ the associated eigenvectors in $S$. The family $(\phi_k)_{k \in \N^*}$ is a Hilbert basis of $L^2$. 
\\
Our goal is to stabilize the ground state. As the global phase of the wave function is physically meaningless, our target set is
\begin{equation}
\label{def_C}
\CC:= \left\{ c\phi \, ; \, c\in \mathbb{C} \text{ and } |c|=1 \right\},
\end{equation}
where $\phi:= \phi_1$.

Let $J_{\neq 0}:= \Big\{ k \geq 2 ; \lag Q_1\phi , \phi_k \rag \neq 0 \Big\}$ and $J_0 := \Big\{ k \geq 2 ; \lag Q_1\phi , \phi_k \rag = 0 \Big\}$. We assume that the following hypotheses hold.

\begin{hypo}
\label{hypo}
\hspace*{5mm}

\begin{enumerate}[i)]
\item $\forall k \in J_0, \lag Q_2\phi , \phi_k\rag \neq 0$ i.e. all coupling are realized either by $Q_1$ or $Q_2$,
\item $Card(J_0) < \infty$ i.e. only a finite number of coupling is missed by $Q_1$,
\item $\lambda_1 - \lambda_k \neq \lambda_p -\lambda_q$ for $k,p,q \geq 1$ such that $\{1,k\} \neq \{p,q\}$ and $k\neq 1$,
\end{enumerate}
\end{hypo}
\begin{rmq}
The hypothesis \textit{i)} is weaker than the one in \cite{BeauchardNersesyan} (i.e. $J_0 = \emptyset$). As proved in \cite[Section 3.4]{Nersesyan}, we get that generically with respect to $Q_1$ and $Q_2$ in $C^{\infty}(\overline{D},\R)$, the scalar products $\lag Q_1 \phi , \phi_k \rag$ and $\lag Q_2 \phi , \phi_k \rag$ are all non zero. The spectral assumption \textit{iii)} does not hold in every physical situation. For example, it is not satisfied in 1D if $V = 0$. However, it is proved in \cite[Lemma 3.12]{Nersesyan} that if $D$ is the rectangle $[0,1]^n$, Hypothesis 1.1 \textit{iii)} hold generically with respect to $V$ in the set $\mathcal{G} := \big\{ V \in C^{\infty}(D,\R) \, ; \, V(x_1,\dots,x_n)=V_1(x_1)+\dots + V_n(x_n), \text{ with } V_k \in C^{\infty}([0,1],\R) \big\}$.
\\
\end{rmq}

As in \cite{CGLT}, we use a time-periodic oscillating control of the form 
\begin{equation}
\label{defu}
u(t,\psi) := \alpha(\psi) + \beta(\psi) \sin \left( \frac{t}{\varepsilon} \right).
\end{equation}
\\
Following classical techniques (see e.g. \cite{SandersVerhulst07}) of dynamical systems in finite dimension let us introduce the averaged system
\begin{equation}
\label{eqaverage2}
\left\{
\begin{aligned}
i \partial_t \psi_{av} &= \big( -\Delta + V(x) \big) \psiav + \alpha(\psiav)Q_1(x) \psiav 
\\
&+ \left( \alpha(\psiav)^2 + \frac{1}{2} \beta(\psiav)^2 \right) Q_2(x) \psiav ,
\\
\psi_{av_{| \partial D}}  &= 0,
\end{aligned}
\right.
\end{equation}
with initial condition
\begin{equation}
\label{CI_psiav}
\psiav(0,\cdot) = \psi^0.
\end{equation}

\noindent
Let $P$ be the orthogonal projection in $L^2$ onto the closure of Span $\{ \phi_k ; k \geq 2 \}$ and $\gamma$ be a positive constant (to be determined later). 
\\
Our stabilization strategy relies on the following Lyapunov function (used in \cite{BeauchardNersesyan}) defined on $S \cap H^1_0 \cap H^2$ by
\begin{equation}
\label{def_Lyapunov}
\mathcal{L}(\psi) := \gamma || (-\Delta + V) P \psi ||_{L^2}^2  + 1  - | \lag \psi , \phi \rag |^2 . 
\end{equation}
This leads to feedback laws given by
\begin{equation}
\label{feedback}
\alpha(\psiav(t,\cdot)) := -k I_1(\psiav(t,\cdot)), \quad \quad \beta(\psiav(t,\cdot)) := g(I_2(\psiav(t,\cdot)),
\end{equation}
with $k>0$ small enough and 
\begin{equation}
\label{defg}
g \in C^2(\mathbb{R}, \mathbb{R}^+) \text{ satisfying } g(x)=0 \text{ if and only if } x\geq 0, \ 
 g' \text{ bounded},
\end{equation} 
and for $j \in \{1,2\}$, for $z \in H^2$,
\begin{equation}
\label{defI}
I_j(z) = \text{Im} \Big[ \gamma \lag (-\Delta +V)P \, Q_j z , (-\Delta+V) P z \rag 
 - \lag Q_j z, \phi\rag \lag \phi,z\rag \Big].
\end{equation}

\noindent
We can now state the well-posedness of the averaged closed loop system (\ref{eqaverage2}).
\begin{prop}
\label{prop_reg}
Let $R>0$. There exists $k_0=k_0 \left( V, Q_2, R \right) > 0$ such that for any $\psi^0 \in H^2 \cap H^1_0 \cap S$ with $\LL(\psi^0) < R$ and $k \in (0,k_0)$, the closed-loop system (\ref{eqaverage2})-(\ref{CI_psiav})-(\ref{feedback}) has a unique solution $\psiav \in C^0([0,+\infty), H^2 \cap H^1_0)$. 
There exists $M > 0$ such that
\begin{equation}
\label{maj_psiav}
|| \psiav(t) ||_{H^2} \leq M, \quad \forall t \geq 0.
\end{equation}
Moreover, if $\Delta \psi^0 \in H^1_0 \cap H^2$, then $\Delta \psiav \in C^0([0,+\infty), H^1_0 \cap H^2)$.
\end{prop}

\noindent
We define $X_0 := \left\{ \psi^0 \in S \cap H^1_0 \cap H^2 ; \Delta \psi^0 \in H^1_0 \cap H^2 \right\}$ the set of admissible initial conditions. For an initial condition $\psi^0 \in X_0$, we define the control 
\begin{equation}
\label{defu_eps}
u^{\varepsilon}(t) := \alpha(\psiav(t)) + \beta(\psiav(t)) \sin \left( \frac{t}{\varepsilon} \right),
\end{equation}
where $\psiav$ is the solution of (\ref{eqaverage2})-(\ref{CI_psiav})-(\ref{feedback}).

The main result of this article is the following one.
\begin{theo}
\label{theo_suite}
Assume that Hypotheses \ref{hypo} hold. Let $\CC$, the target set, be defined by (\ref{def_C}). There exists $k_0=k_0(V,Q_2)>0$ such that for any $k \in k_0$, for any $s<2$ and for any $\psi^0 \in X_0$ with $0< \LL(\psi^0) < 1$, there exist an increasing time sequence $(T_n)_{n\in \mathbb{N}}$ in $\mathbb{R}^*_+$ tending to $+\infty$ and a decreasing sequence $(\varepsilon_n)_{n \in \N}$ in $\R^*_+$ such that if $\psieps$ is the solution of (\ref{syst})-(\ref{syst3}) associated to the control $u^{\varepsilon}$ defined by (\ref{defu_eps}) then for all $n \in \N$, if $\varepsilon \in (0,\varepsilon_n)$,
\begin{equation*}
\text{dist}_{H^s} \left( \psieps(t,\cdot) , \CC \right)  \leq \frac{1}{2^n}, \quad \forall t \in [T_n,T_{n+1}] .
\end{equation*} 
\end{theo}

\begin{rmq}
Theorem \ref{theo_suite} gives the semi-global approximate controllability with explicit controls of system (\ref{syst}). Hypotheses \ref{hypo} are  needed to ensure that the invariant set coincides with the target set. The semi-global aspect comes from the hypothesis $0< \LL(\psi^0) < 1$ : by reducing $\gamma$ (in a way dependant of $\psi^0$), this condition can be fulfilled as soon as $\psi^0 \notin \mathcal{C}$.

In Theorem \ref{theo_suite}, there is a gap between the $H^4$ regularity of the initial condition and the approximate controllability in $H^s$ with $s<2$. The extra-regularity is used in this article to prove an approximation property in $H^2$ between the oscillating system and the averaged one (see Section \ref{sect_appro}). Weakening this regularity assumption is an open problem for which an alternative strategy is required. The last lost of regularity comes from the application of a weak LaSalle principle instead of a strong one due to lack of compactness in infinite dimension. 
\end{rmq}

\subsection{A review of previous results}

$\indent$ In this section, we recall previous results about quantum systems with bilinear controls. The model (\ref{syst}) of an infinite potential well was proposed by Rouchon in \cite{RouchonModele} in the dipolar approximation ($Q_2 =0$). A classical negative result was obtained in \cite{BallMarsdenSlemrod82} by Ball, Marsden and Slemrod for infinite dimensional bilinear control systems. This result implies, for system (\ref{syst}) with $Q_2=0$, that the set of reachable states from any initial data in $H^2 \cap H^1_0 \cap S$ with control in $L^2(0,T)$ has a dense complement in $H^2 \cap H^1_0 \cap S$. However, exact controllability was proved in 1D by Beauchard in \cite{Beauchard05} for $V=0$ and $Q_1(x)=x$ in more regular spaces ($H^7$). This result was then refined in \cite{BeauchardLaurent} by Beauchard and Laurent for more general $Q_1$ and a regularity $H^3$.

The question of stabilization is addressed in \cite{BeauchardNersesyan} where Beauchard and Nersesyan extended previous results from Nersesyan \cite{Nersesyan}. They proved, under appropriate assumptions on $Q_1$, the semi-global weak $H^2$ stabilization of the wave function towards the ground state using explicit feedback control and Lyapunov techniques in infinite dimension.

However sometimes, for example in the case of higher laser intensities, this model is not efficient (see e.g. \cite{Dion_1,Dion_2}) and we need to add a polarizability term $u(t)^2Q_2(x) \psi$ in the model. This term, if not neglected, can also be helpful in mathematical proofs. Indeed the result of \cite{BeauchardNersesyan} only holds if $Q_1$ couples the ground state to any other eigenstate and then the use of the polarizability enables us to weaken this assumption. Mathematical use of the expansion of the Hamiltonian beyond the dipolar approximation was used by Grigoriu, Lefter and Turinici in \cite{GrigoriuLefterTurinici09, Turinici07}. A finite dimension approximation of this model was studied in \cite{CGLT} by Coron, Grigoriu, Lefter and Turinici. The authors proposed discontinuous feedback laws and periodic highly oscillating feedback laws to stabilize the ground state. In this article, we extend in our infinite dimensional framework their idea of using (time-dependent) periodic feedback laws. We also refer to the book \cite{CoronBook} by Coron for a comprehensive presentation of the feedback strategy and the use of time-varying feedback laws.

How to adapt the Lyapunov or LaSalle strategy in an infinite dimensional framework is not clear because closed bounded sets are not compact so the trajectories may lack compactness in the considered topology. In this direction we should cite some related works of Mirrahimi and Beauchard \cite{BeauchardMirrahimi09,Mirrahimi09} where the idea was to prove approximate convergence results. 
In this article, we will use an adaptation of the LaSalle invariance principle for weak convergence which was used for example in \cite{BeauchardNersesyan} by Beauchard and Nersesyan.
There are other strategies to show a strong stabilization property. Coron and d'Andréa-Novel proved in \cite{CAN} the compactness of the trajectories by a direct method for a beam equation and thus the strong stabilization. In \cite{Couchouron1,Couchouron2} Couchouron gave sufficient conditions to obtain the compactness in favorable cases where the control acts diagonally on the state. Another strategy to obtain strong results is to look for a strict Lyapunov function, which is an even trickier question, and was done for example in \cite{CoronAndreaNovelBastin07} by Coron, d'Andréa-Novel and Bastin for a system of conservation laws.

The question of approximate controllability has been addressed by various authors using various techniques. In \cite{Nersesyan10}, Nersesyan uses a Lyapunov strategy to obtain approximate controllability in large time in regular spaces. In \cite{CMSB09}, Chambrion, Mason, Sigalotti and Boscain proved approximate controllability in $L^2$ for a wider class of systems using geometric control tools for the Galerkin approximations. The hypotheses needed were weakened in \cite{BCCS11} and the approximate controllability was extended to some $H^s$ spaces in \cite{BoussaidCaponigroChambrion}.

Explicit approximate controllability in large time has also been obtained by Ervedoza and Puel in \cite{ErvedozaPuel09} on a model of trapped ion, using different tools.

\subsection{Structure of this article}

$\indent$ As announced in Section \ref{sect_main_result}, we study the system (\ref{syst}) by introducing a highly oscillating time-periodic control and the corresponding averaged system. Section \ref{sect_lyapu} is devoted to the introduction of this averaged system and its weak stabilization using Lyapunov techniques and an adaptation of the LaSalle invariance principle in infinite dimension.

In Section \ref{sect_appro} we study the approximation property between the solution of the averaged system and the solution of (\ref{syst}) with the same initial condition. We prove that on every finite time interval these two solutions remain arbitrarily close provided that the control is oscillating enough. This is an extension of classical averaging results for finite dimension dynamical systems.

Finally gathering the stabilization result of Section \ref{sect_lyapu} and the approximation property of Section \ref{sect_appro}, we prove Theorem \ref{theo_suite} in Section \ref{sect_result}.

Section \ref{sect_simulations} is devoted to numerical simulations illustrating several aspects of Theorem \ref{theo_suite} and of the averaging strategy.


\section{Stabilization of the averaged system}
\label{sect_lyapu}

\subsection{Definition of the averaged system}

$\indent$ System (\ref{syst}) with feedback law $u$ defined by (\ref{defu}) can be rewritten as
\begin{equation}
\label{syst4}
\left\{
\begin{aligned}
\partial_t \psi(t) &= A \psi(t) + F \left( \frac{t}{\varepsilon},\psi(t) \right), 
\\
\psi_{| \partial D} &= 0,
\end{aligned}
\right. 
\end{equation}
where the operator $A$ is defined by $D(A) := H^2 \cap H^1_0$,  $ A \psi := ( i\Delta - iV) \psi$ and
\begin{equation}
\label{defF}
F(s,z):= -i \left( \alpha(z) + \beta(z) \sin(s) \right) Q_1 z - i \left( \alpha(z) + \beta(z) \sin(s) \right)^2 Q_2 z .
\end{equation}

For any $z$, $F(.,z)$ is $T$-periodic (with here $T=2\pi$). Following classical techniques of averaging, we introduce $F^0(z):= \frac{1}{T} \int_0^{T} F(t,z) \md t$. We can define the averaged system associated to (\ref{syst4}) by 
\begin{equation}
\label{eqaverage}
\left\{
\begin{aligned}
\partial_t \psi_{av} &= A\psiav + F^0(\psiav),
\\
\psi_{av _{| \partial D}} &=0 .
\end{aligned}
\right.
\end{equation}
Straightforward computations of $F^0$ show that the system (\ref{eqaverage}) can be rewritten as (\ref{eqaverage2}).

We show by Lyapunov techniques that we can choose $\alpha$ and $\beta$ such that the solution of the averaged system (\ref{eqaverage}) is weakly convergent in $H^2$ towards our target set $\CC$.

\subsection{Control Lyapunov function and damping feedback laws}

$\indent$ Our candidate for the Lyapunov function, $\LL$, is defined in (\ref{def_Lyapunov}).
It is clear that $ \mathcal{L} (\psi) \geq 0$ whenever $\psi \in S \cap H^1_0 \cap H^2 $ and that $\mathcal{L}(\psi) = 0$ if and only if $\psi \in \mathcal{C} $.

The main advantage of this Lyapunov function is that it can be used to bound the $H^2$ norm. In fact,  for any $\psi \in S \cap H^1_0 \cap H^2$,
\begin{equation*}
\mathcal{L} (\psi) 
\geq \gamma || (-\Delta +V) P \psi ||_{L^2}^2
\geq \frac{\gamma}{2} || \Delta(P\psi) ||_{L^2}^2 - C
\geq \frac{\gamma}{4} || \Delta \psi ||_{L^2}^2 - C ,
\end{equation*}
where here, as in all this article, $C$ is a positive constant possibly different each time it appears. This leads to the existence of $\tilde{C} >0$ satisfying
\begin{equation}
\label{borneh2}
|| \psi ||_{H^2}^2 \leq \tilde{C} ( 1+\mathcal{L}(\psi) ), \quad \forall\psi \in S \cap H^1_0 \cap H^2.
\end{equation}

\begin{rmq}
Although the idea of using a feedback of the form (\ref{defu}) is inspired by \cite{CGLT}, the construction of the Lyapunov function and of the controls is here different because we are dealing with an infinite dimensional framework. We follow the strategy used in \cite{Nersesyan,BeauchardNersesyan}.
\end{rmq}

\paragraph*{Choice of the feedbacks.}

We would like to choose the feedbacks $\alpha$ and $\beta$ such that for all $t \geq 0$, $\dfrac{d}{dt} \mathcal{L} (\psiav(t)) \leq 0$ where $ \psiav $ is the solution of (\ref{eqaverage2}),(\ref{CI_psiav}).

If $\Delta \psiav(t) \in H^1_0 \cap H^2$ for all $t\geq 0$ then
\begin{align*}
\frac{\mathrm d}{\mathrm dt} & \mathcal{L}(\psiav(t))
= 2 \gamma \text{Re} \big[ \lag (-\Delta+V) P \partial_t \psi_{av} , (-\Delta +V) P\psiav\rag \big]   
\\
&- 2\text{Re} \big[ \lag \partial_t \psi_{av} , \phi\rag \lag \phi , \psiav\rag \big] 
\\
&= 2 \gamma \text{Re} \Big[ \lag (-\Delta +V) P \big(i\Delta \psiav - iV\psiav - i\alpha Q_1 \psiav 
\\
&- i( \alpha^2 + \frac{1}{2} \beta ^2 ) Q_2 \psiav \big) , (-\Delta +V ) P\psiav \rag \Big] 
\\
&  -2 \text{Re} \Big[ \lag i\Delta \psiav - iV\psiav - i\alpha Q_1 \psiav 
- i( \alpha^2 + \frac{1}{2} \beta ^2 ) Q_2 \psiav , \phi \rag \lag \phi ,\psiav\rag \Big].
\end{align*}
\\
Then we perform integration by parts. As $P$ commutes with $(-\Delta +V)$, $V$ is real and thanks to the following boundary conditions
\begin{equation*}
(-\Delta+V)P\psi_{av | \partial D} = \psi_{av | \partial D} = \phi_{| \partial D} = 0,
\end{equation*}
\\
we have 
\begin{align*}
& 2 \gamma \text{Re} \Big[ \lag -i (-\Delta +V)^2 P\psiav , (-\Delta+V)P\psiav \rag \Big] \\
&- 2 \text{Re} \Big[ \lag (i\Delta -iV)\psiav,\phi \rag  \lag\phi ,\psiav\rag \Big]  \\
&= 2\gamma \text{Re} \Big[ \lag -i \nabla (-\Delta+V) P\psiav, \nabla(-\Delta +V) P\psiav\rag \Big]  \\
&+ 2\gamma \text{Re} \Big[ \lag -i V(-\Delta + V)P\psiav, (-\Delta + V) P\psiav \rag \Big] \\
&+ 2\lambda_1 \text{Re}  \Big[ \lag i\psiav ,\phi \rag \lag\phi , \psiav\rag \Big] \\
&= 0.
\end{align*}
\\
This leads to
\begin{equation}
\label{lyapunov}
\frac{\mathrm d}{\mathrm dt} \mathcal{L}(\psiav(t)) = 2 \alpha I_1(\psiav(t)) + 2 \left(\alpha^2 + \frac{1}{2} \beta^2 \right) I_2(\psiav(t)),
\end{equation}
where  $I_j$ is defined in (\ref{defI}).

\noindent
In order to have a decreasing Lyapunov function we define the feedback laws $\alpha$ and $\beta$ as in (\ref{feedback}).
Thus (\ref{lyapunov}) becomes 
\begin{equation}
\label{lyapunov2}
\frac{\mathrm d}{\mathrm dt} \mathcal{L} (\psiav (t)) = -2 \left( kI_1^2(1-kI_2) - \frac{1}{2} I_2g^2(I_2) \right).
\end{equation}

If we assume that we can choose the constant $k$ such that $(1-kI_2) > 0$ for all $t\geq 0$ and if $ \Delta \psiav(t) \in  H^1_0 \cap H^2$ then the feedbacks (\ref{feedback}) in system (\ref{eqaverage2}) lead to
\begin{equation}
\label{lyapunov3}
\frac{\mathrm d}{\mathrm dt} \mathcal{L}(\psiav (t)) \leq 0 , \quad \forall t \geq 0 .
\end{equation}

\paragraph*{Well-posedness and boundedness proofs.}

Using the previous heuristic on the Lyapunov function, we can state and prove the well-posedness of the closed loop system (\ref{eqaverage2})-(\ref{feedback}) globally in time and derive a uniform bound on the $H^2$ norm of the solution. Namely, we prove Proposition \ref{prop_reg}.

\begin{proof}[Proof of Proposition \ref{prop_reg}.]

By the explicit expression (\ref{defI}) of $I_2$, we get for any $z \in H^2$, $|I_2(z)| \leq f \big( ||z||_{H^2} \big)$ where 
\begin{equation*}
f(x) := ||Q_2||_{L^{\infty}} + \gamma ( x + ||V||_{L^{\infty}} + \lambda_1) 
(||Q_2||_{\mathcal{C}^2} x + 
 ||V||_{L^{\infty}} ||Q_2||_{L^{\infty}} + \lambda_1 ||Q_2||_{L^{\infty}}) .
\end{equation*}
Notice that $f$ is increasing on $\mathbb{R}^+$.
Let $K:= 2f \Big( \sqrt{\tilde{C} (1+R)} \Big)$ where $\tilde{C}$ is defined by (\ref{borneh2}), $k_0:=\frac{1}{K}$ and $k \in \left( 0 , k_0 \right)$.

The local existence and regularity is obtained by a classical fixed point argument : there exists $T^*>0$ such that the closed loop system (\ref{eqaverage2}) with initial condition (\ref{CI_psiav}) and feedback laws (\ref{feedback}) admits a unique solution defined on $(0,T^*)$ and satisfying either $T^*= + \infty$ or $T^* < + \infty$ and 
\begin{equation*}
\underset{t \to T^*}{\limsup} || \psiav(t) ||_{H^2} = + \infty .
\end{equation*}

We have
\begin{equation*}
|I_2(\psiav(0))| \leq f(||\psi^0||_{H^2}) \leq f \left( \sqrt{\tilde{C}(1+ \LL(\psi^0))} \right) \leq \frac{K}{2},
\end{equation*}
thus, by continuity, $|I_2(\psiav(t))| \leq K$ for $t$ small enough.

\noindent
Let
\begin{equation*}
T_{max} := \sup \left\{ t \in (0,T^*) ; |I_2(\psiav(\tau))| \leq K, \forall \tau \in (0,t) \right\}.
\end{equation*}
We want to prove that $T_{max}=T^*=+\infty$.

\noindent
For all $t \in [0, T_{max})$, we have  $\big( 1-kI_2(\psiav(t)) \big) > 0$, which implies (by (\ref{lyapunov2})), $\mathcal{L}(\psiav(\cdot))$ is decreasing on $[0,T_{max})$. 
Estimate (\ref{borneh2}) leads to 
\begin{equation}
\label{maj_temp_psiav}
||\psiav(t)||_{H^2} \leq \sqrt{\tilde{C} (1+\mathcal{L}(\psiav(t))} \leq \sqrt{\tilde{C}(1+\mathcal{L}(\psi^0))}, \quad \forall t \in [0,T_{max}).
\end{equation}
Let us proceed by contradiction and assume that $T_{max} < T^*$. This implies $|I_2(\psiav(T_{max}))|=K$. By definition of $K$,
\begin{equation*}
|I_2(\psiav(t))| \leq f \left( \sqrt{ \tilde{C} (1+ \mathcal{L}(\psi^0) } \right) \leq \frac{K}{2} \quad \forall t \in [0,T_{max}) .
\end{equation*}
This is inconsistent with $|I_2(\psiav(T_{max}))|=K$ so $T_{max}=T^*$ and the solution is bounded in $H^2$ when it is defined. As no blow-up is possible thanks to (\ref{maj_temp_psiav}) we obtain that $T_{max}=T^*=+\infty$ and thus the solution is global in time and bounded.
\\
Finally, taking the time derivative of the equation we obtain the announced regularity.
\end{proof}

\subsection{Convergence Analysis}

In all this section we assume that $k \in (0,k_0)$ where $k_0$ is defined in Proposition \ref{prop_reg} with $R=1$.
The closed-loop stabilization for the averaged system (\ref{eqaverage2}) is given by the next statement.

\begin{theo}
\label{convergence}
Assume that Hypotheses \ref{hypo} hold. If $\psi^0 \in X_0$ with $0< \LL(\psi^0) < 1$,  then the solution $\psiav$ of the closed-loop system (\ref{eqaverage2})-(\ref{feedback}) with initial condition (\ref{CI_psiav}) satisfies
\begin{equation*}
\psiav(t) \underset{t \to \infty}{\rightharpoonup} \CC
\quad \text{ in } H^2 .
\end{equation*}
\end{theo}

We prove this theorem by adapting the LaSalle invariance principle to infinite dimension in the same spirit as in \cite{BeauchardNersesyan}. This is done in two steps. First we prove that the invariant set, relatively to the closed-loop system (\ref{eqaverage2})-(\ref{feedback}) and the Lyapunov function $\LL$, is $\CC$. Here, Hypothesis \ref{hypo} is crucial. Then we prove that every adherent point for the weak $H^2$ topology of the solution of this closed-loop system is contained in $\CC$. This is due to the continuity of the propagator of the closed-loop system for the weak $H^2$ topology.

\subsubsection{Invariant set}
\begin{prop}
\label{ens_invariant}
Assume that Hypotheses \ref{hypo} hold. Assume that $\psi^0$ belongs to $S\cap H^1_0 \cap H^2$ and satisfies $\lag \psi^0, \phi \rag \neq 0$. If the function $t \mapsto \mathcal{L} (\psiav(t))$ is constant, then $\psi^0 \in \mathcal{C}$.
\end{prop}

\begin{proof}
Thanks to (\ref{lyapunov2}), the fact that $\big( 1-kI_2(\psiav(t)) \big)> 0$ for all $t \geq 0$ and (\ref{defg}) we get 
\begin{equation*}
I_1[\psiav(\cdot)] \equiv 0 , \quad I_2(\psiav(\cdot))g^2 \left( I_2(\psiav(\cdot)) \right) \equiv 0 \, \text{ i.e. } 
I_2(\psiav(t)) \geq 0 , \quad \forall t \geq 0.
\end{equation*}
By (\ref{feedback}) this implies that $\alpha(\psiav(\cdot)) \equiv \beta(\psiav(\cdot)) \equiv 0$ and then $\psiav$ is solution of the uncontrolled Schr\"odinger equation. So,
\begin{equation*}
\psiav(t)= \sum_{j=1}^{\infty} e^{-i\lambda_j t} \lag \psi^0 , \phi_j\rag \phi_j .
\end{equation*}
Recall that $\phi := \phi_1$ is the ground state. Following the idea of \cite{Nersesyan}, we obtain after computations and gathering the terms with different exponential term 
\begin{align*}
I_1(\psiav(t)) & =
\sum_{j , k \geq 2} \tilde{P}(\psi^0,j,k,Q_1) e^{-i(\lambda_j - \lambda_k)t} 
+  \sum_{j \in J_{\neq 0}} \tilde{\tilde{P}}(\psi^0,j,Q_1) e^{i(\lambda_j - \lambda_1)t} \\
& + \sum_{j \in J_{\neq 0}} \lag\psi^0,\phi_j\rag \lag\phi,\psi^0\rag  \lag Q_1\phi_j,\phi\rag ( 1+ \gamma \lambda_j^2) e^{-i(\lambda_j - \lambda_1)t} ,
\end{align*}
\\
where $\tilde{P}(\psi^0,j,k,Q_1)$ and $\tilde{\tilde{P}}(\psi^0,j,Q_1)$ are constants. Then, by \cite[Lemma 3.10]{Nersesyan},
\begin{equation*}
\lag\psi^0,\phi_j\rag \lag\phi,\psi^0\rag \lag Q_1\phi_j,\phi\rag ( 1+ \gamma \lambda_j^2) = 0 , \quad \forall j\in J_{\neq 0}. 
\end{equation*}
Using the assumption $\lag \phi , \psi^0 \rag \neq 0$ and Hypotheses \ref{hypo} it comes that for all $j \in J_{\neq 0}$, $\lag \psi^0,\phi_j \rag =0$. This leads to 
\begin{equation*}
\psiav(t)= e^{-i\lambda_1 t} \lag \psi^0 , \phi\rag \phi + \sum_{j\in J_0} e^{-i\lambda_j t} \lag \psi^0 , \phi_j\rag \phi_j ,
\end{equation*}
where by Hypotheses \ref{hypo}, $J_0$ is a finite set. By simple computations we obtain,
\begin{align*}
I_2(\psiav(t)) &= \text{Im} \Big( 
\sum\limits_{k,j \in J_0} \gamma \lambda_j \lag\phi_j , \psi^0\rag \lag\psi^0 , \phi_k\rag \lag(-\Delta+V) P (Q_2 \phi_k), \phi_j\rag e^{i(\lambda_j - \lambda_k)t} \\
&+ \sum\limits_{j \in J_0} \gamma \lambda_j \lag\phi_j , \psi^0\rag \lag\psi^0 , \phi\rag \lag(-\Delta+V) P (Q_2 \phi), \phi_j\rag e^{i(\lambda_j - \lambda_1)t} \\
&- \sum\limits_{j\in J_0} \lag\psi^0 , \phi_j\rag \lag\phi , \psi^0\rag \lag Q_2 \phi_j , \phi\rag e^{-i(\lambda_j - \lambda_1)t} \\
&- | \lag\psi^0 , \phi\rag |^2 \lag Q_2\phi , \phi\rag \Big) \geq 0 .
\tag{\theequation} \addtocounter{equation}{1} \label{defLambda}
\end{align*}
There exists $N_0 \in \N^*$ and $( \mu_n)_{n \in \{ 0 , \dots , N_0 \} }$ such that
\begin{equation*}
\left\{ \mu_n \, ; \, n \in \{ 0 , \dots , N_0 \} \right\} =
\left\{ \pm (\lambda_k - \lambda_j) \, ; \, (k,j) \in J_0 \times (J_0 \cup \{1\}) \right\},
\end{equation*}
with $\mu_0 = 0$ and $\mu_j \neq \mu_k$ if $j \neq k$. Thus, (\ref{defLambda}) implies that for any $n \in \{ 0 , \dots , N_0 \} $, there exists $\Lambda_n = \Lambda_n(\psi^0,Q_2) \in \C$ such that
\begin{equation}
\label{I2}
\text{Im} \Big( \sum\limits_{j=0}^{N_0} \Lambda_j e^{i \mu_j t} \Big) \geq 0, \quad \forall t\geq 0.
\end{equation}
Straightforward computations give
\begin{equation*}
\Lambda_0 = \sum_{j \in J_0} \Big( \gamma \lambda_j^2  |\lag\phi_j , \psi^0\rag|^2  \lag Q_2 \phi_j , \phi_j\rag  \Big)
- |\lag\psi^0,\phi\rag|^2 \lag Q_2\phi, \phi\rag .
\end{equation*}
\\
Thus, $\text{Im} \big( \Lambda_0  \big)=0 $ and our inequality (\ref{I2}) can be rewritten as
\begin{equation*}
\text{Im} \Big( \sum\limits_{j=1}^{N_0} \Lambda_j e^{i \mu_j t} \Big) \geq 0, \quad \forall t\geq 0,
\end{equation*}
with the $\mu_j$ being all different and non-zero. Then using the same argument as in \cite[Proof of Theorem 3.1]{CGLT}, we get that $\Lambda_j = 0$ for $j\geq 1$ and then using (\ref{defLambda}) in particular that the coefficient of $e^{-i(\lambda_j-\lambda_1)t}$ vanishes. It implies $\lag \psi^0 , \phi_j\rag = 0$ for all $j \in J_0$.
Consequently, $ \psi^0 = \lag\psi^0,\phi\rag \phi$. As $\psi^0, \phi \in S$, we obtain $\psi^0 \in \mathcal{C}$.
\end{proof}

\subsubsection{Weak $H^2$ continuity of the propagator}

$\indent$ We denote by $\mathcal{U}_t(\psi^0)$ the propagator of the closed-loop system (\ref{eqaverage2})-(\ref{feedback}). We detail here the continuity property of this propagator and of the feedback laws we need to apply the LaSalle invariance principle.

\begin{prop}
\label{continuite_c.i.}
Let $z_n \in S \cap H^1_0 \cap H^2$ be a sequence such that $ z_n \rightharpoonup z_{\infty}$ in $H^2$.
For every $T>0$, there exists $N \subset (0,T)$ of zero Lebesgue measure verifying for all $t \in (0,T) \backslash N$,
\begin{enumerate}[i)]
\item $\displaystyle{\mathcal{U}_t(z_n) \underset{n \to \infty}{\rightharpoonup} \mathcal{U}_t(z_{\infty})}$ in $H^2$,
\item $\displaystyle{\alpha \big( \mathcal{U}_t(z_n) \big) \underset{n \to \infty}{\rightarrow} \alpha \big( \mathcal{U}_t(z_{\infty}) \big)} $ and $ \displaystyle{\beta \big( \mathcal{U}_t(z_n) \big) \underset{n \to \infty}{\rightarrow} \beta \big( \mathcal{U}_t(z_{\infty}) \big)}$.
\end{enumerate}
\end{prop}

\begin{proof}

\noindent \emph{Proof of ii).}
We start by proving that if $(z_n)_{n\in \N} \in H^1_0 \cap H^2$ satisfy 
\\
$\displaystyle{z_n \underset{n \to \infty}{\rightharpoonup} z_{\infty}}$ in $H^2$ then $\displaystyle{ \alpha( z_n ) \underset{n \to \infty}{\rightarrow} \alpha( z_{\infty} )} $ and $\displaystyle{\beta( z_n ) \underset{n \to \infty}{\rightarrow} \beta( z_{\infty} )}$. Thus $ii)$ will be a simple consequence of $i)$.
As proved in \cite[Proposition 2.2]{BeauchardNersesyan}, using the fact that the regularity $H^{3/2}$ is sufficient to define the feedback, we get
\begin{equation*}
I_j(z_n) \underset{ n \rightarrow + \infty}{\longrightarrow} I_j(z_{\infty}),\quad \text{for } j=1,2.
\end{equation*}
So by the design of our feedback,
\begin{equation*}
\alpha(z_n) \underset{ n \rightarrow + \infty}{\longrightarrow}\alpha (z_{\infty}) , \quad 
\beta(z_n) \underset{ n \rightarrow + \infty}{\longrightarrow} \beta (z_{\infty}).
\end{equation*}

\noindent \emph{Proof of i).}
The exact same proof as in \cite[Proposition 2.2]{BeauchardNersesyan} based on extraction in less regular spaces, uniqueness property of the closed loop system and taking into account the polarizability term leads to the announced result.
\end{proof}

\subsubsection{LaSalle invariance principle}

$\indent$ We now have all the needed tools to prove Theorem \ref{convergence}.

\begin{proof}[Proof of Theorem \ref{convergence}]
Consider $\psi^0 \in X_0$ with $0 < \mathcal{L}(\psi^0) <1$.
Thanks to the bound (\ref{borneh2}), $\mathcal{U}_{t}(\psi^0)$ is bounded in $H^2$. Let $(t_n)_{n \in \N}$ be a sequence of times tending to $+\infty$ and $\psi_{\infty} \in H^2$ be such that  
$\mathcal{U}_{t_n}(\psi^0) \underset{n \to \infty}{\rightharpoonup} \psi_{\infty}$ in $H^2$. We want to show that $\psi_{\infty} \in \mathcal{C}$.

We prove that $\alpha(\mathcal{U}_t(\psi_{\infty})) =0$ and $\beta(\mathcal{U}_t(\psi_{\infty}))=0$.
Indeed, the function $t\mapsto \alpha \big( \mathcal{U}_t(\psi^0) \big)$ belongs to $L^2(0,+\infty)$ (because of (\ref{lyapunov2}) and (\ref{feedback})) so the sequence of functions $(t\in(0,+\infty) \mapsto  \alpha \big( \mathcal{U}_{t_n+t}(\psi^0) \big)_n$ tends to zero in $L^2(0,+\infty)$. Then by the Lebesgue reciprocal theorem there exists a subsequence $(t_{n_k})_{k \in \N}$ and $N_1 \subset (0,+\infty)$ of zero Lebesgue measure such that
\begin{equation*}
\alpha \big( \mathcal{U}_{t+t_{n_k}} (\psi^0) \big) \underset{k \to \infty}{\rightarrow} 0 , \quad \forall t \in (0,+\infty) \backslash N_1.
\end{equation*}

Let $T \in (0, +\infty)$. Using Proposition \ref{continuite_c.i.}, there exists $N \subset (0,T)$  of zero Lebesgue measure such that
\begin{equation*}
\alpha \big( \mathcal{U}_{t+t_{n_k}} (\psi^0) \big) \underset{k \to \infty}{\rightarrow} \alpha \big( \mathcal{U}_t (\psi_{\infty}) \big) ,
\quad \forall t \in (0,T) \backslash N.
\end{equation*}
Hence, $\alpha \big( \mathcal{U}_t (\psi_{\infty}) \big) = 0$ for all $t\in (0,T) \backslash (N_1 \cup N)$. The function $t \mapsto \alpha \big( \mathcal{U}_{t} (\psi_{\infty}) \big)$ being continuous we get 
$\alpha \big( \mathcal{U}_t (\psi_{\infty}) \big) = 0$ for all $t\in [0,T]$, and this for all $ T>0$. Finally 
$\alpha \big( \mathcal{U}_t (\psi_{\infty}) \big) = 0$ for all $t \geq 0$.

The same argument holds for $\beta$ as $\tilde{g} : t\mapsto I_2 \big( \mathcal{U}_t(\psi^0) \big) g^2 \left( I_2 \big( \mathcal{U}_t(\psi^0) \big) \right)$ belongs to $L^1(0,+\infty)$. Then by the proof of Proposition \ref{continuite_c.i.},
\begin{equation*}
\tilde{g} \big( \mathcal{U}_{t+t_{n_k}} (\psi^0) \big) \underset{k \to \infty}{\rightarrow} \tilde{g} 
\big( \mathcal{U}_t (\psi_{\infty}) \big) ,
\quad \forall t \in (0,T) \backslash N,
\end{equation*}
and $\tilde{g} \big( \mathcal{U}_t (\psi_{\infty}) \big) =0$ implies 
$\beta \big( \mathcal{U}_t (\psi_{\infty}) \big) \equiv 0$.

These two results lead to the fact that $\mathcal{L} \big( \mathcal{U}_t (\psi_{\infty}) \big)$ is constant.
\\
By (\ref{lyapunov3}), $\mathcal{L}(\psi_{\infty}) \leq \mathcal{L}(\psi^0) < 1$ so $\lag\psi_{\infty},\phi\rag \neq 0$. All assumptions of Proposition \ref{ens_invariant} are satisfied then $\psi_{\infty} \in \mathcal{C}$.
\\
This concludes the proof of Theorem \ref{convergence} and the convergence analysis of (\ref{eqaverage2}).
\end{proof}

\section{Approximation by averaging}
\label{sect_appro}

$\indent$ The method of averaging was mostly used for finite-dimensional dynamical systems (see e.g. \cite{SandersVerhulst07}). The concept of averaging in quantum control theory has already produced interesting results. For example, in \cite{MirrahimiSarletteRouchon10} the authors make important use of these averaging properties in finite dimension through what is called in quantum physics the rotating wave approximation. The main idea of using a highly oscillating control is that if it is oscillating enough the initial system behaves like the averaged system. We extend this concept in our infinite dimensional framework : we prove an approximation result on every finite time interval.
More precisely we have the following result.
\begin{prop}
\label{avgprop}
Let $[s,L]$ be a fixed interval and $\psi^0 \in X_0$ with $0 < \LL(\psi^0) <1$. Let $\psiav$ be the solution of the closed loop system (\ref{eqaverage2}),(\ref{feedback}) with initial condition $\psiav(s,\cdot) = \psi^0$. For any $\delta>0$, there exists $\varepsilon_0 >0$ such that, if  $\psieps$ is the solution of (\ref{syst}) associated to the same initial condition $\psieps(s,\cdot) = \psi^0$ and control $u^{\varepsilon}(t)$ defined by (\ref{defu_eps}) with $\varepsilon \in (0,\varepsilon_0)$ then
\begin{equation*}
|| \psieps(t,\cdot) - \psiav(t,\cdot) ||_{H^2} \leq \delta , \quad \forall t \in [s,L].
\end{equation*}
\end{prop}

\begin{rmq} Notice that the controls $\alpha$ and $\beta$ were defined using the averaged system in a feedback form but the control $u^{\varepsilon}$ used for the system (\ref{syst}) is explicit and is not defined as a feedback control.
\end{rmq}

\begin{rmq} Due to the infinite dimensional framework, we are facing regularity issues and cannot adapt directly the strategy of \cite{SandersVerhulst07}.
\end{rmq}

\begin{proof}
We define for $(t,z,\tilde{z}) \in \R \times H^2 \times H^2$,
\begin{equation}
\label{defFopen}
\tilde{F}(t,z,\tilde{z}) := -i \left( \alpha(\tilde{z}) + \beta(\tilde{z}) \sin(t) \right) Q_1 z - i \left( \alpha(\tilde{z}) + \beta(\tilde{z}) \sin(t) \right)^2 Q_2 z .
\end{equation}
Notice that thanks to (\ref{defF}) for any $(t,z) \in \R \times H^2$,
\begin{equation}
\label{F_Ftilde}
\tilde{F}(t,z,z) = F(t,z).
\end{equation}
With these notations the considered system (\ref{syst}) with control (\ref{defu_eps}) and initial condition $\psieps(s,\cdot)= \psi^0$  can be rewritten as
\begin{equation*}
\left\{
\begin{aligned}
\partial_t \psieps(t) &= A \psieps(t) + \tilde{F} \left( \frac{t}{\varepsilon} , \psieps(t) , \psiav(t) \right) ,
\\
\psi_{\varepsilon_{| \partial D}} &= 0 ,
\end{aligned}
\right.
\end{equation*}
where $\psiav$ is the solution of the closed-loop system (\ref{eqaverage2}) with initial condition $\psiav(s,\cdot)=\psi^0$. 
\\
Denoting by $T_A$ the semigroup generated by $A$, we have for any $t \geq s$,
\begin{align*}
\psieps(t) &= T_A(t-s) \psi^0 + \int_s^t T_A(t-\tau) \tilde{F} \left( \frac{\tau}{\varepsilon} , \psieps(\tau) , \psiav(\tau) \right) \md \tau,  
\\
\psiav(t) &= T_A(t-s) \psi^0 + \int_s^t T_A(t-\tau) F^0 \big(  \psiav(\tau) \big) \md \tau.
\end{align*}
This implies for any $t \geq s$,
\begin{equation}
\label{demo_avg1}
\begin{split}
||  \psieps(t) - \psiav(t) ||_{H^2}  \leq 
 \Big| \Big| \int_s^t T_A(t-\tau)  \left[ F \left( \frac{\tau}{\varepsilon} , \psiav(\tau) \right) - F^0 \left(  \psiav(\tau) \right) \right] \md \tau \Big| \Big|_{H^2} 
\\
+  \Big| \Big| \int_s^t T_A(t-\tau) \Big[  \tilde{F} \left( \frac{\tau}{\varepsilon} , \psieps(\tau), \psiav(\tau) \right) 
- F \left( \frac{\tau}{\varepsilon} , \psiav(\tau) \right) \Big] \md \tau \Big| \Big|_{H^2} .
\end{split}
\end{equation}
\\
We study separately the two terms of the right-hand side of (\ref{demo_avg1}).

\textit{First step : }
We show the existence of $C>0$ such that for any $t \geq s$, for any $\varepsilon > 0$,
\begin{equation}
\label{avg_term1}
\begin{split}
\Big| \Big|  \int_s^t T_A(t-\tau) \left[ \tilde{F} \left( \frac{\tau}{\varepsilon} , \psieps(\tau) , \psiav(\tau) \right) - F \left( \frac{\tau}{\varepsilon} , \psiav(\tau) \right) \right] \md \tau \Big| \Big|_{H^2} 
\\
 \leq C \int_s^t \big| \big| \psieps(\tau) - \psiav(\tau) \big| \big|_{H^2} \md \tau .
\end{split}
\end{equation}

By (\ref{defF}),(\ref{defFopen}), it comes that for any $\tau \geq s$, for any $\varepsilon > 0$,
\begin{equation*}
\begin{split}
\tilde{F} & \left( \frac{\tau}{\varepsilon} , \psieps(\tau), \psiav(\tau) \right) - F \left( \frac{\tau}{\varepsilon} , \psiav(\tau) \right) 
\\
&= -i \left( \alpha(\psiav(\tau)) + \beta(\psiav(\tau)) \sin \left( \frac{\tau}{\varepsilon} \right) \right)  Q_1 \left[ \psieps(\tau) - \psiav(\tau) \right]
\\
&-i \left(  \alpha (\psiav(\tau)) + \beta(\psiav(\tau)) \sin \left( \frac{\tau}{\varepsilon} \right) \right)^2  Q_2 \left[ \psieps(\tau) - \psiav(\tau) \right].
\end{split}
\end{equation*}
As $\psiav$ is bounded in $H^2$,  using (\ref{feedback}) and (\ref{defI}) we get the existence of $M_1>0$ such that for all $\tau \geq s $, 
\begin{equation}
\label{borne_feedback}
|\alpha(\psiav( \tau ))| + |\beta(\psiav(\tau))|  \leq  M_1 .
\end{equation}
As $| \sin \left( \frac{\tau}{\varepsilon} \right) | \leq 1$, we get the existence of $C>0$ independent of $\varepsilon$ such that for any $\tau \geq s$, for any $\varepsilon > 0$,
\begin{equation}
\label{lipsF}
\Big| \Big| \tilde{F} \left( \frac{\tau}{\varepsilon} , \psieps(\tau), \psiav(\tau) \right) - F \left( \frac{\tau}{\varepsilon} , \psiav(\tau) \right) \Big| \Big|_{H^2} 
\leq C \big| \big| \psieps(\tau) - \psiav(\tau) \big| \big|_{H^2} .
\end{equation}
Then the contraction property of $T_A$ implies (\ref{avg_term1}).


\textit{Second step : }
We show that there exists $C>0$ satisfying for all $t \in [s,L]$, for any $\varepsilon > 0$, 
\begin{equation}
\label{avg_term2}
\Big| \Big|  \int_s^t  T_A(t-\tau) \left[ F \left( \frac{\tau}{\varepsilon} ,\psiav(\tau) \right) - F^0(\psiav(\tau)) \right] \md \tau \Big| \Big|_{H^2} \leq C \varepsilon .
\end{equation}

We follow computations on the semigroup $T_A$ done in \cite{HaleLunel90}.
For $(t,v) \in \mathbb{R}^+ \times C^1([s,L],H^2)$, we define $U$ and $H$ by 
\begin{align*}
U(t,v(\cdot)) :&= \int_0^t \big( F(\tau,v(\cdot)) - F^0(v(\cdot)) \big) \md \tau, 
\\
H(t,v) :&= \md_v U(t,v) \dot{v} ,
\end{align*}
where $\dot{v}$ is the time derivative of $v$.

\noindent
Notice that the $T$-periodicity of $F(\cdot,v)$ and the definition of $F^0$ imply that $U(\cdot,v)$ is also $T$-periodic.

\begin{lemme}
\label{lemme_AvgInf}
As $\psiav \in C^1([s,L],H^1_0 \cap H^2)$, we have for any $t \in [s,L]$, 
for any $\varepsilon > 0$,
\begin{align*}
\int_s^t & T_A(t-\tau) \left[ F \left( \frac{\tau}{\varepsilon} ,\psiav(\tau) \right) 
- F^0(\psiav(\tau)) \right] \md \tau = 
\\
&\varepsilon U \left( \frac{t}{\varepsilon} , \psiav(t) \right) - \varepsilon T_A(t-s) U \left( \frac{s}{\varepsilon},\psiav(s) \right) 
\\
+ &\varepsilon A \int_s^t T_A(t-\tau) U \left( \frac{\tau}{\varepsilon},\psiav(\tau) \right) \md \tau 
- \varepsilon \int_s^t T_A(t-\tau) H \left( \frac{\tau}{\varepsilon} , \psiav(\tau) \right) \md \tau .
\end{align*}
\end{lemme}

\begin{proof}
The proof is done in \cite[Lemma 2.2]{HaleLunel90} .
\end{proof}

\noindent
We study separately each term of the previous right-hand side.

\noindent
$\bullet$
With $\kappa= \lfloor \frac{t}{\varepsilon T} \rfloor$, we have $\frac{t}{\varepsilon} - \kappa T \in [0,T]$ and by periodicity 
\begin{align*}
U \left( \frac{t}{\varepsilon} , \psiav(t) \right) &= \int_0^{t/\varepsilon} \Big( F(\tau,\psiav(t)) - F^0(\psiav(t)) \Big) \md \tau \\
&= \int_0^{t/\varepsilon -\kappa T} \Big( F(\tau,\psiav(t)) - F^0(\psiav(t)) \Big) \md \tau.
\end{align*}
\\
As $\psiav$ is bounded in $H^2$ and $\alpha(\psiav)$, $\beta(\psiav)$ are bounded there exists $M_2 >0$ such that
\begin{equation*}
|| F(\tau, \psiav(t) )||_{H^2} \leq M_2, 
\quad || F^0(\psiav(t)) ||_{H^2} \leq M_2,
\quad \forall \tau \geq 0, \forall t \geq s.
\end{equation*}
This leads to 
\begin{equation*}
\Big| \Big| U \left( \frac{t}{\varepsilon} , \psiav(t) \right) \Big| \Big|_{H^2}
\leq \int_0^{t/\varepsilon -\kappa T} 2M_2 \md \tau 
\leq 2 M_2 T,  \quad \forall t \geq s, \forall \varepsilon > 0.
\end{equation*} 
\\
The same computations lead to 
\begin{equation*}
\Big| \Big| T_A(t-s) U \left( \frac{s}{\varepsilon} , \psiav(s) \right) \Big| \Big|_{H^2}
\leq 2 M_2 T, 
\quad \forall t \geq s, \forall \varepsilon > 0.
\end{equation*} 
\\
Then,
\begin{equation}
\label{avg_term2_1}
\Big| \Big|  \varepsilon U \left( \frac{t}{\varepsilon} , \psiav(t) \right) + \varepsilon T_A(t-s) U \left( \frac{s}{\varepsilon} , \psiav(s) \right) \Big| \Big|_{H^2} \leq C \varepsilon, 
\quad \forall t \geq s, \forall \varepsilon>0.
\end{equation}

\noindent
$\bullet$
By switching property, 
\begin{equation*}
A \int_s^t T_A (t-\tau) U \left( \frac{\tau}{\varepsilon} ,\psiav(\tau) \right) \md \tau
= \int_s^t T_A (t-\tau) A U \left( \frac{\tau}{\varepsilon} ,\psiav(\tau) \right) \md \tau ,
\end{equation*}
and for any $t \in [s,L]$, for any $\varepsilon > 0$
\begin{align*}
A U\left( \frac{t}{\varepsilon} , \psiav(t) \right) &= 
A \int_0^{t/\varepsilon - \kappa T} \left[ F( \tau ,\psiav(t)) - F^0(\psiav(t)) \right] \md \tau \\
&= \int_0^{t/\varepsilon - \kappa T} \left[ A F( \tau ,\psiav(t)) - A F^0(\psiav(t)) \right] \md \tau .
\end{align*}

\noindent
By definition of $F$ and $F^0$ we have
\begin{align*}
A F(t,z) &= -i(\alpha(z) + \beta(z) \sin(t)) A (Q_1z) - i(\alpha(z) + \beta(z) \sin(t))^2 A(Q_2 z),
\\
A F^0(z) &= -i \alpha(z) A (Q_1z) - i \left( \alpha(z)^2 + \frac{1}{2} \beta(z)^2 \right) A(Q_2z) .
\end{align*}

\noindent
By regularity hypothesis on $Q_1$, $Q_2$ and $V$ there exists $C>0$ such that
\begin{equation*}
|| A (Q_1 z) ||_{H^2} \leq C ||\Delta z||_{H^2}, \quad ||A (Q_2 z)||_{H^2} \leq C || \Delta z||_{H^2} .
\end{equation*}
Thus thanks to Proposition \ref{prop_reg} and the bound (\ref{borne_feedback}) on $\alpha(\psiav)$ and $\beta(\psiav)$,  we get the existence of $M_3 >0$ satisfying
\begin{equation*}
|| A F(\tau,\psiav(t)) ||_{H^2} \leq M_3 , \quad ||A F^0(\psiav(t))||_{H^2} \leq M_3, \quad \forall \tau \geq 0, \forall t \in [s,L].
\end{equation*}

\noindent
So, for any $t \in [s,L]$, for any $\varepsilon >0$, $ \Big| \Big| A U \big( \frac{t}{\varepsilon},\psiav(t) \big) \Big| \Big|_{H^2} \leq 2M_3T$.
\noindent
Consequently , there exists $C>0$ such that
\begin{equation}
\label{avg_term2_2}
\Big| \Big| \varepsilon A \int_s^t T_A (t-\tau) U \left( \frac{\tau}{\varepsilon} ,\psiav(\tau) \right) \md \tau  \Big| \Big|_{H^2}
\leq C \varepsilon, \quad \forall t \in [s,L], \forall \varepsilon >0 .
\end{equation}

$\bullet$
For the last term we need to estimate $ H\left( \frac{t}{\varepsilon} ,\psiav(t) \right)$. We have
\begin{align*}
H \left( \frac{t}{\varepsilon} ,\psiav(t) \right) 
&= \md_v U \left( \frac{t}{\varepsilon} , \psiav(t) \right) . \partial_t \psi_{av}(t) \\
&= \int_0^{t/\varepsilon - \kappa T} \Big( \md_v F(\tau , \psiav).\partial_t \psi_{av} - \md F^0(\psiav).\partial_t \psi_{av} \Big) \md \tau .
\end{align*}

\noindent
Using (\ref{feedback}) and (\ref{defI}), we have for any $v,w \in C^0([s,L], H^1_0 \cap H^2)$, 
\begin{equation}
\label{diff_a_b}
\md \alpha(v).w = -k \md I_1(v).w , \quad
\md \beta (v).w = g' \big( I_2(v) \big) \md I_2(v).w,
\end{equation}
where, 
\begin{align*}
\md I_j(v).w & = \text{Im} \Big[ \gamma \lag (-\Delta +V)P(Q_j w) , (-\Delta+V) Pv \rag \\
&+ \gamma \lag (-\Delta +V)P(Q_j v) , (-\Delta+V) Pw \rag \\
& - \lag Q_j w, \phi\rag \lag \phi,v\rag - \lag Q_j v, \phi\rag \lag \phi,w\rag \Big].
\end{align*}

Finally, we have
\begin{align}
\notag
\md_v F(t,v).w &= -i(\alpha(v) + \beta(v) \sin t) Q_1 w - i( \md \alpha(v).w + \md \beta(v).w \sin t) Q_1 v 
\\
\label{diff_F}
- i (\alpha(v) + \beta(v) & \sin t)^2  Q_2w  - 2i (\alpha(v) + \beta(v) \sin t) ( \md \alpha(v).w + \md \beta(v).w \sin t) Q_2 v ,
\end{align}
and
\begin{align}
\notag
\md F^0(v).w &= -i \alpha(v) Q_1 w - i \md \alpha(v).w Q_1 v - i \left( \alpha(v)^2 + \frac{1}{2} \beta(v)^2 \right) Q_2 w 
\\
\label{diff_F0}
-i (2 \alpha(v) &\md \alpha(v).w  + \beta(v) \md \beta(v).w \sin t) Q_2v .
\end{align}

By Proposition \ref{prop_reg}, $\partial_t \psi_{av} \in C^0([0,+\infty), H^1_0 \cap H^2)$ so there exists $M_4 > 0$ such that
\begin{equation*}
|| \partial_t \psi_{av} (t) ||_{H^2} \leq M_4, \quad \forall t \in [s,L].
\end{equation*}
Hence the same computations as previously lead to the existence of $C>0$ satisfying
\begin{equation*}
| \md \alpha(\psiav(t)). \partial_t \psi_{av}(t) | + | \md \beta(\psiav(t)). \partial_t \psi_{av}(t) | \leq C ,
 \quad \forall t \in [s,L],
\end{equation*}
and thus by (\ref{diff_F}),(\ref{diff_F0}),  for any $t \in [s,L]$, for any $\tau \geq 0$,
\begin{equation*}
|| \md_v F(\tau ,\psiav(t)).\partial_t \psi_{av}(t) ||_{H^2} + || \md F^0(\psiav(t)).\partial_t \psi_{av}(t) ||_{H^2} 
\leq C .
\end{equation*}
As a consequence,
\begin{equation*}
\Big| \Big| H \left( \frac{\tau}{\varepsilon},\psiav(\tau) \right) \Big| \Big|_{H^2} \leq CT,
\quad \forall \tau \in [s,L], \forall \varepsilon>0,
\end{equation*}
\\
and then,
\begin{equation}
\label{avg_term2_3}
\Big| \Big| \varepsilon \int_s^t T_A(t-\tau) H \left( \frac{\tau}{\varepsilon},\psiav(\tau) \right) \Big| \Big|_{H^2} \md \tau
\leq (CLT) \varepsilon.
\end{equation}

We are now able to deal with the remaining term of the right-hand side of (\ref{demo_avg1}).
Gathering inequalities (\ref{avg_term2_1}), (\ref{avg_term2_2}) and (\ref{avg_term2_3}) in Lemma \ref{lemme_AvgInf} we obtain that there exists $C > 0$ such that inequality (\ref{avg_term2}) holds.

\textit{Third step : }
Putting together (\ref{demo_avg1}), (\ref{avg_term1}) and (\ref{avg_term2}) we obtain that there exists $C>0$ such that for any $t \in [s,L]$, for any $\varepsilon>0$,
\begin{equation*}
|| \psieps(t) - \psiav(t) ||_{H^2} \leq  C \varepsilon + C \int_s^t ||\psieps(\tau) - \psiav(\tau) ||_{H^2} \md \tau .
\end{equation*}
\\
Hence Gr\"onwall's lemma implies
\begin{equation*}
|| \psieps(t) - \psiav(t) ||_{H^2} \leq C \varepsilon e^{C(t-s)} \leq (C e^{C(L-s)}) \varepsilon,
\quad \forall t \in [s,L],
\end{equation*}
and Proposition \ref{avgprop} is proved with $\displaystyle{ \varepsilon_0 = \dfrac{\delta}{C e^{C(L-s)}} }$ .
\end{proof}

\begin{rmq}
The proof we used is fundamentally based on the boundedness of $\Delta \psiav(t)$ on $[s,L]$ and on Gr\"onwall's lemma so it cannot be extended directly to an infinite time interval $[s,+\infty)$.
\end{rmq}

\section{Explicit approximate controllability}
\label{sect_result}

$\indent$ The solution $\psiav$ of the averaged system (\ref{eqaverage2}),(\ref{feedback}), can be driven in the $H^2$ weak topology to the target set $\mathcal{C}$. 
The solution $\psieps$ of the system (\ref{syst}) associated to the same initial condition, with control $u^{\varepsilon}$, stays close to $\psiav$ on every finite time interval provided that the control is oscillating enough. 
Gathering these two results we prove Theorem \ref{theo_suite}.

\begin{proof}[Proof of Theorem \ref{theo_suite}]
We consider $s<2$ fixed.

By Theorem \ref{convergence}, we can construct an increasing time sequence $(T_n)_{n \in \N}$ tending to $+\infty$ such that for any $n \in \N$,
\begin{equation}
\label{conclu_stab}
\text{dist}_{H^s}(\psiav(t) ,\CC) \leq \frac{1}{2^{n+1}}, \quad \forall t \geq T_n .
\end{equation}

Using Proposition \ref{avgprop} on the time interval $[0,T_{n+1}]$ we then construct a decreasing sequence $(\varepsilon_n)_{n \in \N}$ such that for any $n \in \N$,
\begin{equation}
\label{conclu_avg}
|| \psieps(t) - \psiav(t) ||_{H^s} \leq \frac{1}{2^{n+1}}, \quad \forall t \in [0,T_{n+1}], \forall \varepsilon \in (0, \varepsilon_n) .
\end{equation}

Then (\ref{conclu_stab}),(\ref{conclu_avg}) imply that
\begin{equation*}
\forall n \in \N, \quad
\text{dist}_{H^s}(\psieps(t), \CC) \leq \frac{1}{2^n}, \quad \forall t \in [T_n, T_{n+1}], \forall \varepsilon \in (0, \varepsilon_n),
\end{equation*}
which is the statement of Theorem \ref{theo_suite}.
\end{proof}

\section{Numerical simulations}
\label{sect_simulations}

This section is dedicated to numerical simulations of system (\ref{syst}). First, we detail how we approximate the solutions of (\ref{syst}) and (\ref{eqaverage2}). Then, we check the validity of the implemented code. Finally, we illustrate different aspects of Theorem \ref{theo_suite} and of the averaging property, Proposition \ref{avgprop}.

\subsection{Settings}

In all what follows, we set $D=[0,1]$. As the potential $V$ will vary in this section, the eigenelements of $-\Delta+V$ are denoted $\varphi_{k,V}$ and $\lambda_{k,V}$. Any function $\psi \in L^2((0,1),\C)$ is approximated by its first $M$ modes
\begin{equation*}
\psi(t) \approx \sum_{k=1}^M x_k(t) \varphi_{k,V.}
\end{equation*}
The unknown eigenvectors $\varphi_{k,V}$ are approximated in the following way
\begin{equation*}
\varphi_{k,V} \approx \sum_{j=1}^N a_j^k \varphi_{k,0}.
\end{equation*}
The equality $(-\Delta + V) \varphi_{k,V} = \lambda_{k,V} \varphi_{k,V}$ leads to $B a^k = \lambda_{k,V} a^k$ with
\begin{equation*}
a^k = (a^k_1, \dots, a^k_N)^t, 
\quad
B = \text{diag} (\lambda_{1,0} , \dots, \lambda_{N,0}) + \big( \lag V \varphi_{i,0}, \varphi_{j,0} \rag \big)_{1 \leq i, j \leq N}.
\end{equation*}
Notice that $\lambda_{k,0} = (k \pi)^2$ and $\varphi_{k,0} = \sqrt{2} \sin(k \pi \cdot)$ are explicit.
The scalar products are approximated by the Matlab function {\tt quadl}. The eigenelements $a^k$ and $\lambda_{k,V}$ are then approximated by the Matlab function {\tt eig}.

\subsection{Approximation of $\psieps$ and $\psiav$}

Let
\begin{equation*}
H_0 := \text{diag} (\lambda_{1,V} , \dots , \lambda_{M,V}),
\quad
H_n := \big( \lag Q_n \varphi_{i,V} , \varphi_{j,V} \rag \big)_{1 \leq i,j \leq M}, \; n \in \{1,2\}.
\end{equation*}
It follows that the feedback laws (\ref{defI}) are approximated, for $X \in \R^M$ and $j \in \{1,2\}$, by
\begin{align*}
I_j(X) :&= \text{Im} \Big( \gamma \big( H_0 (0 , (H_j X)_2 , \dots , (H_j X)_M)^t \big)^t \, \big( H_0 ( 0 , \overline{x_2}, \dots , \overline{x_M}) \big) 
\\
&- (H_j X)_1 \overline{x_1} \Big),
\end{align*}
leading to
\begin{equation*}
\alpha(X) := -k I_1(X), \quad \beta(X) := - \min(I_2(X), 0).
\end{equation*}
Thus, if we define $X_{av}$, $X_{\varepsilon} \in \R^M$, systems (\ref{syst}) and (\ref{eqaverage2}) are approximated by
\begin{equation}
\label{syst_discret}
i \frac{\md}{\md t} X_{\varepsilon} = \big( H_0 + u_{\varepsilon}(t) H_1 + u_{\varepsilon}(t)^2 H_2 \big) X_{\varepsilon},
\end{equation}
and
\begin{equation}
\label{syst_av_discret}
i \frac{\md}{\md t} X_{av} = \big( H_0 + \alpha( X_{av}) H_1 + ( \alpha^2(X_{av}) + \frac{1}{2} \beta(X_{av})^2 ) H_2 \big) X_{av},
\end{equation}
where $u_{\varepsilon}(t) = \alpha( X_{av}(t)) + \beta(X_{av}(t)) \sin(t/\varepsilon)$. Equations (\ref{syst_discret}) and (\ref{syst_av_discret}) are solved numerically (simultaneously) using Euler method with a time step $\md t$ and a Strang splitting method.

\subsection{Validation}

We now prove the validity of the implemented code. The eigenvectors $\varphi_{k,V}$ are approximated by $N=50$ modes. We take, as a test case, $V(x) := (x-1/2)^2$, $Q_1(x) := x^2$ and $Q_2(x) := x$. The considered initial condition is $\psi^0 =  \frac{1}{\sqrt{2}} \varphi_{1,V} + \frac{i}{\sqrt{2}} \varphi_{2,V}$. The value of the oscillating parameter is $\varepsilon = 10^{-3}$. The parameter $\gamma$ is chosen such that $\LL(\psi^0) = 3/4$. We compute the discrete Lyapunov function for the averaged system and the $H^s$ norm (with $s=1.8$) to the ground state for both the oscillating and the averaged system. The time scale is $[0,T]$ with $T=1000$ and a time step $\md t=10^{-3}$. For $M=5$, we get the results presented in Figure \ref{fig_validation_code}.
\begin{figure}[H]
\centering
\includegraphics[width=5cm]{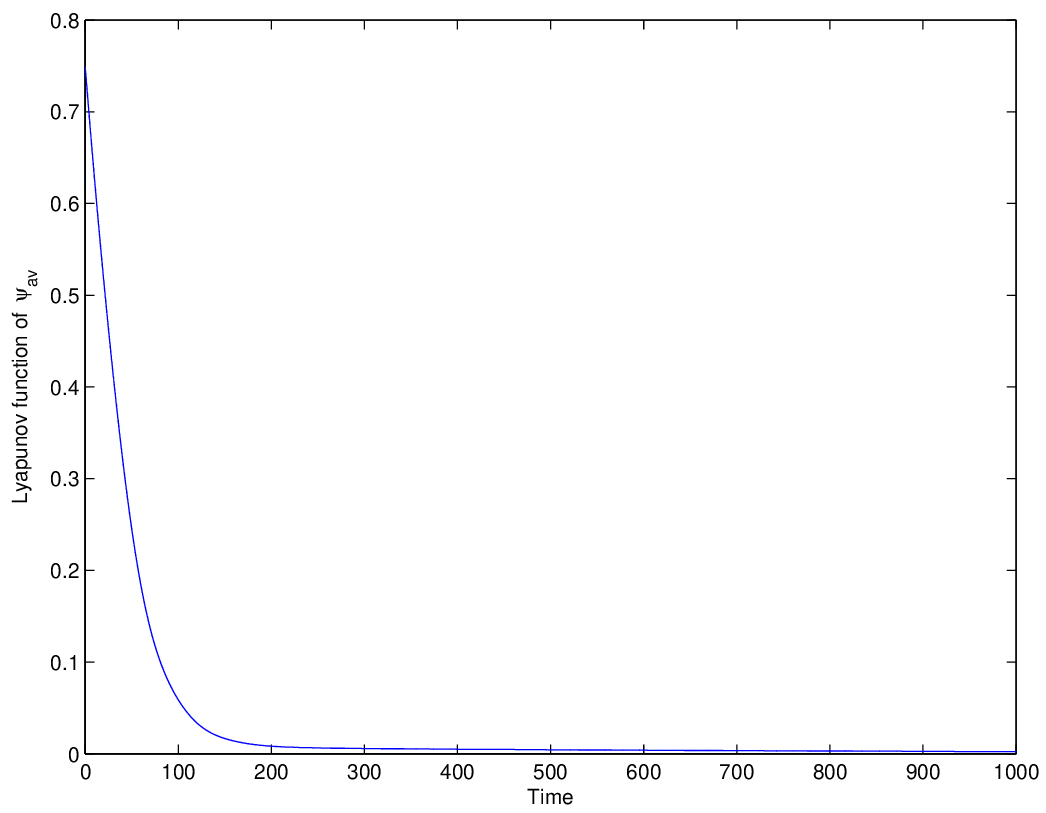}
\hspace{1cm}
\includegraphics[width=5cm]{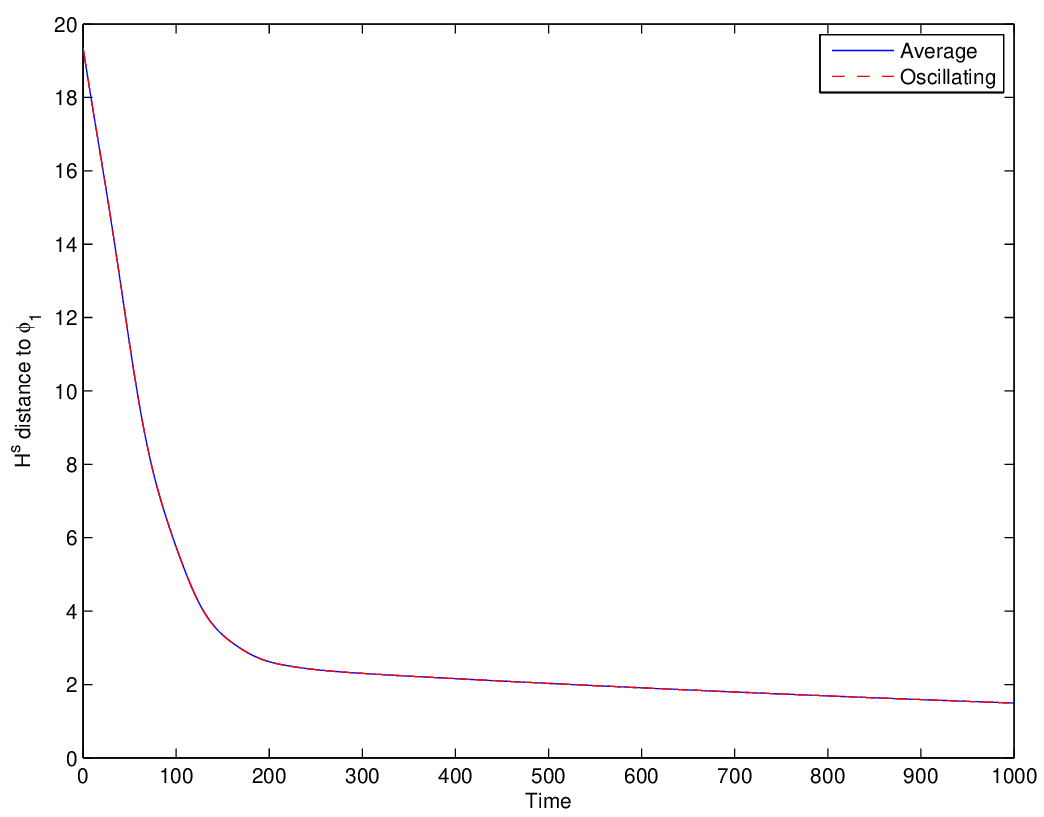}
\caption[]
{\label{fig_validation_code} Lyapunov function of the averaged system (left). $H^s$ norm to the ground state
(right) for the averaged system (continuous line) and the oscillating system (dashed line).}
\end{figure}
\noindent
As expected, we observe the convergence of the Lyapunov function to $0$. The solutions of (\ref{syst_discret}) and (\ref{syst_av_discret}) are driven to the ground state (up to a global phase).
To validate the simulations, we have also tested the code for $M=10$ and $M=20$. We obtained the same asymptotic behaviour and the same values for the Lyapunov function and the $H^s$ distance to the target. 

\noindent
As the approximate controllability uses the fact that the controls are oscillating, the time step $\md t$ cannot be taken large with respect to the oscillating parameter $\varepsilon$. For $\varepsilon=10^{-3}$, we obtain the same results with $\md t = 10^{-3}$ and $\md t = 10^{-4}$. However, instabilities appear on the oscillating system for $\md t= 10^{-2}$. Thus, in all what follows the time step will be chosen smaller than $\varepsilon$. We now present several simulations to illustrate various aspects of Theorem \ref{theo_suite}.

\subsection{Influence of the initial condition}

For every other initial condition tested, the asymptotic behaviour is the same. We present here the results for the same parameters as in Figure \ref{fig_validation_code} but with the initial condition $\psi^0 = \frac{1}{\sqrt{3}} \varphi_{1,V} + \frac{1}{\sqrt{3}} \varphi_{2,V} + \frac{i}{\sqrt{3}} \varphi_{3,V}$. In this case, the stabilization of the averaged system is slower and we computed it for $T=5000$.
\begin{figure}[H]
\centering
\includegraphics[width=5cm]{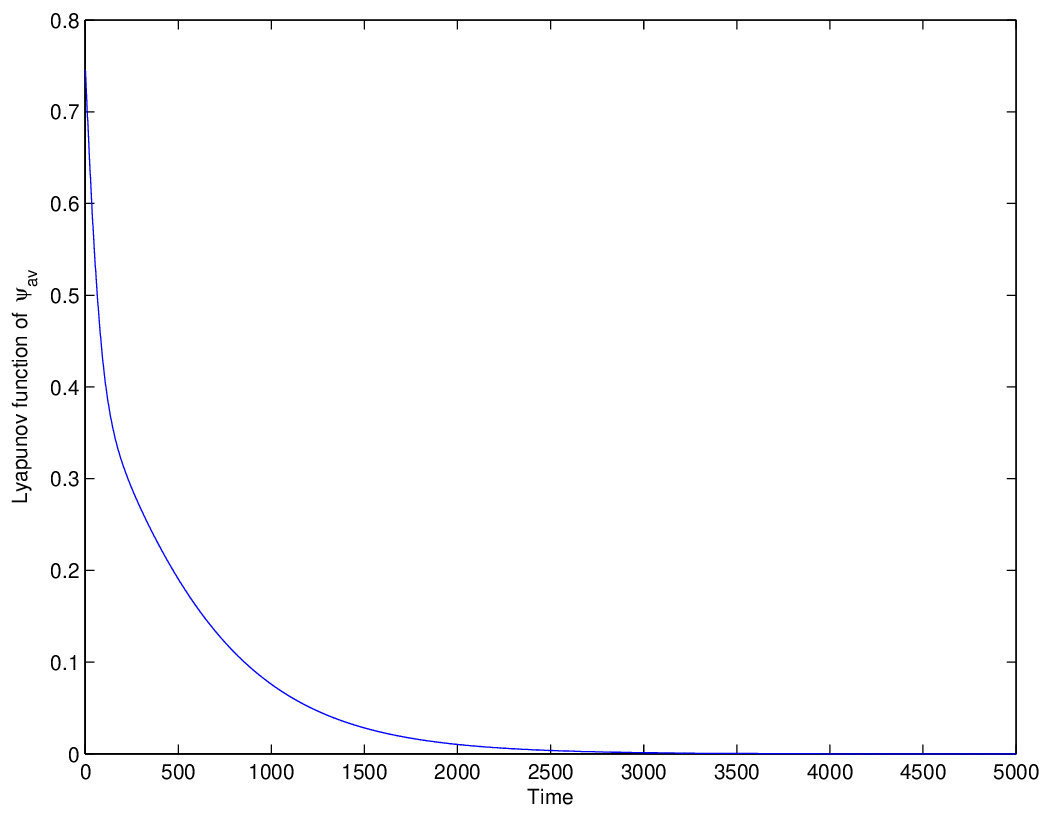}
\hspace{1cm}
\includegraphics[width=5cm]{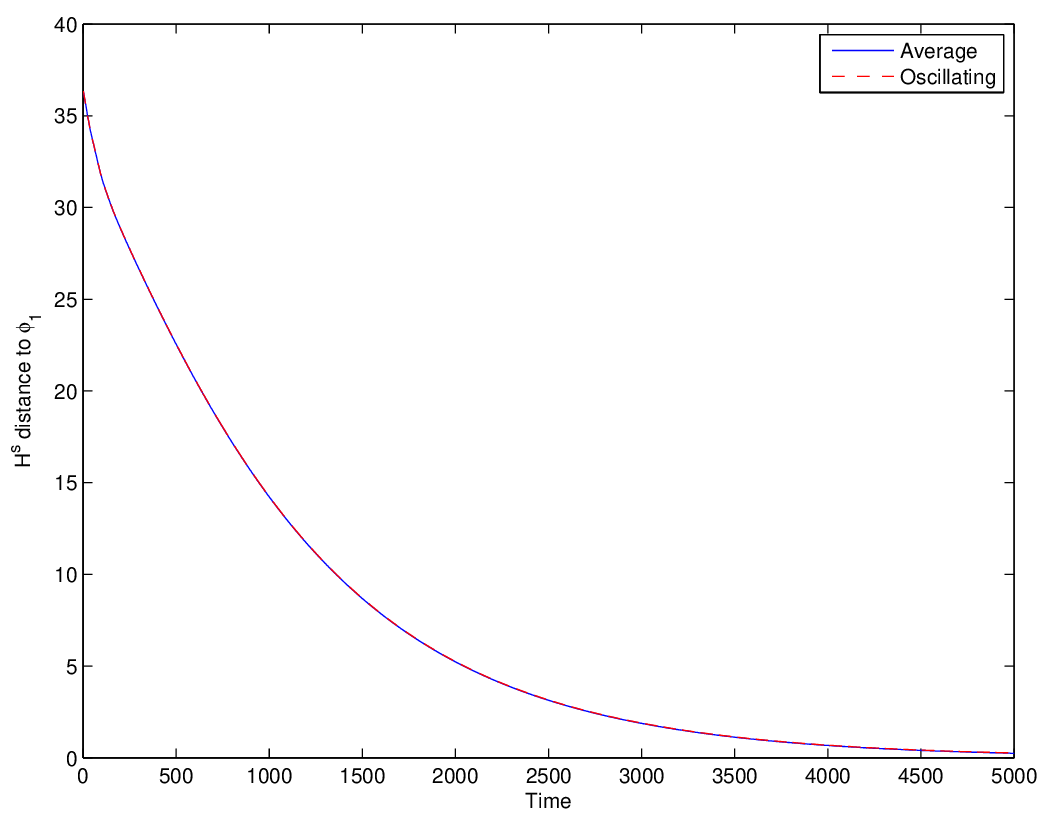}
\caption[]
{\label{fig_cond_ini} Lyapunov function of the averaged system (left). $H^s$ norm to the ground state
(right) for the averaged system (continuous line) and the oscillating system (dashed line).}
\end{figure}
We observe the same asymptotic behaviour as in Figure \ref{fig_validation_code}.

\subsection{Averaging strategy}

We present numerically the influence of the oscillating parameter $\varepsilon$. First, we consider the same potential, dipolar and polarizability moments as in Figure \ref{fig_validation_code}. We compute the discrete $H^s$ norm (for $s=1.8$) to the ground state (up to a global phase) and the discrete $H^2$ norm of $X_{av} - X_{\varepsilon}$. Figure \ref{fig_epsilon_1} is obtained with $\varepsilon = 10^{-3}$ while Figure \ref{fig_epsilon_2} is obtained with $\varepsilon = 10^{-4}$. Both are computed with a time step $\md t = \varepsilon$ and final time $T = 500$.
\begin{figure}[H]
\centering
\includegraphics[width=5cm]{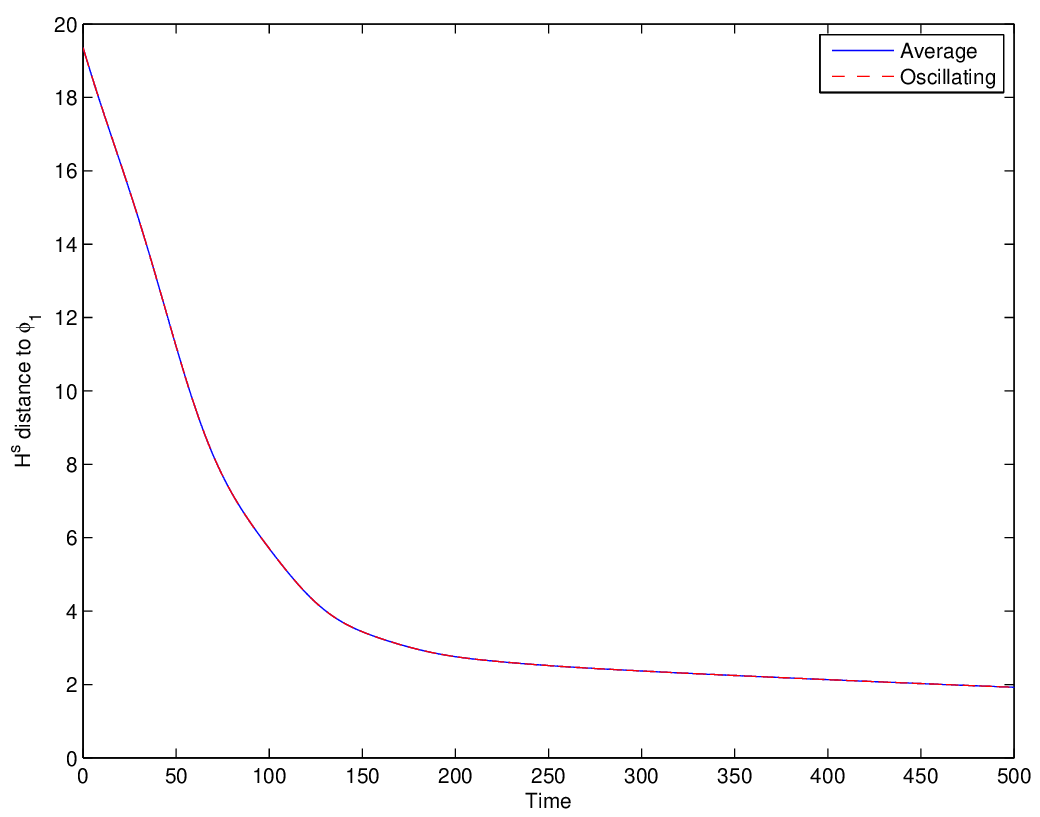}
\hspace{1cm}
\includegraphics[width=5cm]{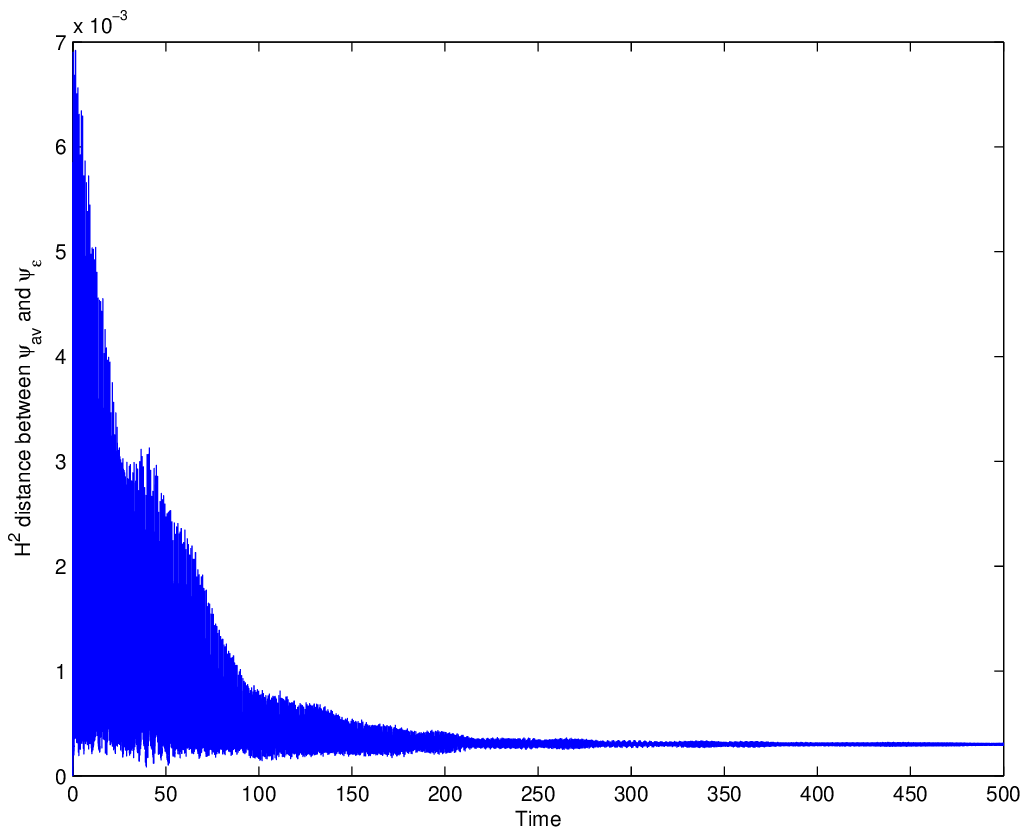}
\caption[]
{\label{fig_epsilon_1} $H^s$ norm to the ground state
(left) for the averaged system (continuous line) and the oscillating system (dashed line). $H^2$ gap from the average (right).}
\end{figure}
\begin{figure}[H]
\centering
\includegraphics[width=5cm]{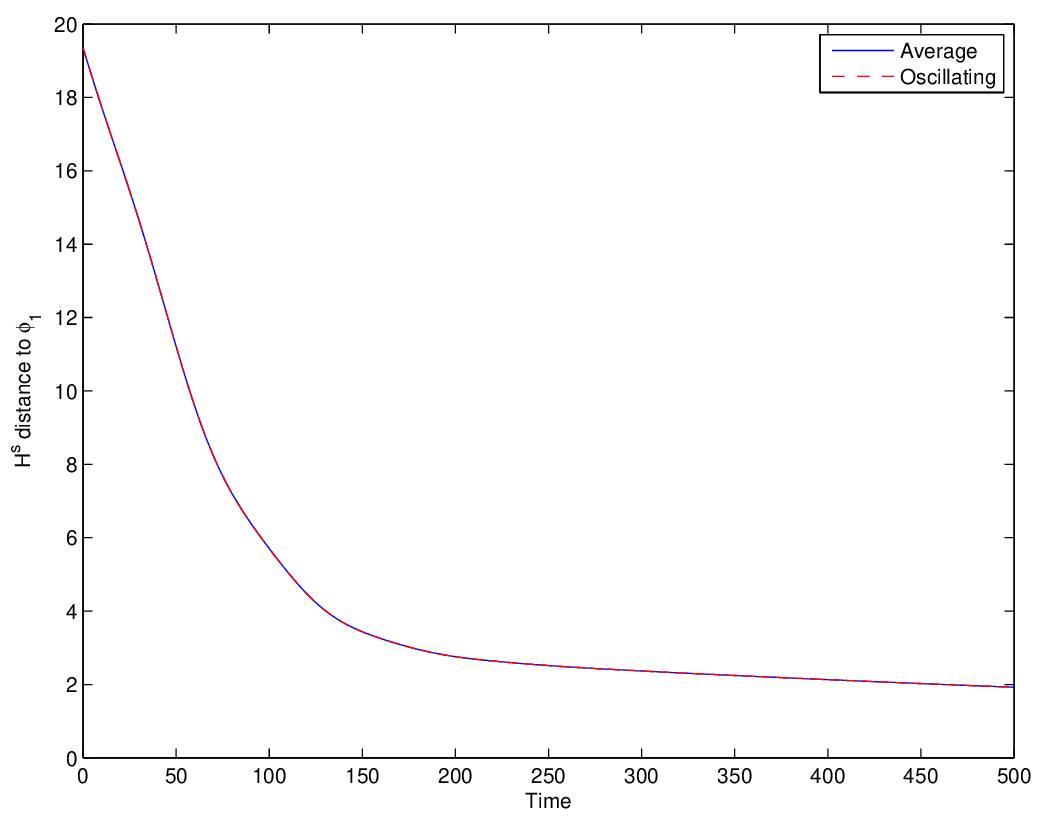}
\hspace{1cm}
\includegraphics[width=5cm]{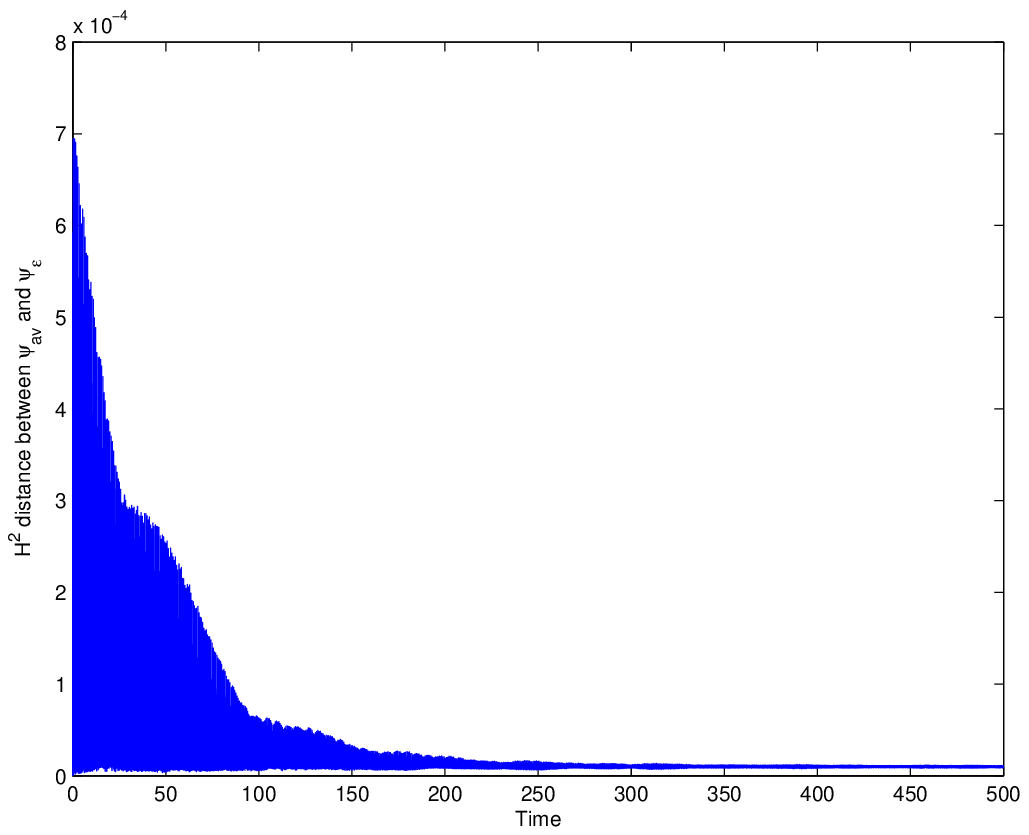}
\caption[]
{\label{fig_epsilon_2} $H^s$ norm to the ground state
(left) for the averaged system (continuous line) and the oscillating system (dashed line). $H^2$ gap from the average (right).}
\end{figure}
For a fixed parameter $\varepsilon$, we observe that the $H^2$ distance between the solution of (\ref{syst}) and the solution of (\ref{eqaverage2}) with the same initial condition does not increase as the time goes to infinity but rather tends to a limit value. This limit value is of the same order of magnitude as $\varepsilon$. We observe that
\begin{equation*}
\frac{||X_{av}(T) - X_{10^{-3}}(T)||_{H^2}}{||X_{av}(T) - X_{10^{-4}}(T)||_{H^2}} \approx 30.
\end{equation*}
This validates numerically the results of Proposition \ref{avgprop} and indicates that this averaging property should be valid on an infinite time horizon.

The same behaviour has been obtained with other parameters. We present here the simulations with $Q_1(x) := cos(x)$ and $Q_2(x) := cos(2x)$, inspired by the physical situation of alignment dynamic of a HCN molecule as in \cite{Dion_2}. Figure \ref{fig_epsilon_3} is obtained with $\varepsilon = 10^{-3}$ while Figure \ref{fig_epsilon_4} is obtained with $\varepsilon = 10^{-4}$. Both are computed with a time step $\md t = \varepsilon$ and final time $T = 1000$ (as the stabilization process seems slower in this case).
\begin{figure}[H]
\centering
\includegraphics[width=5cm]{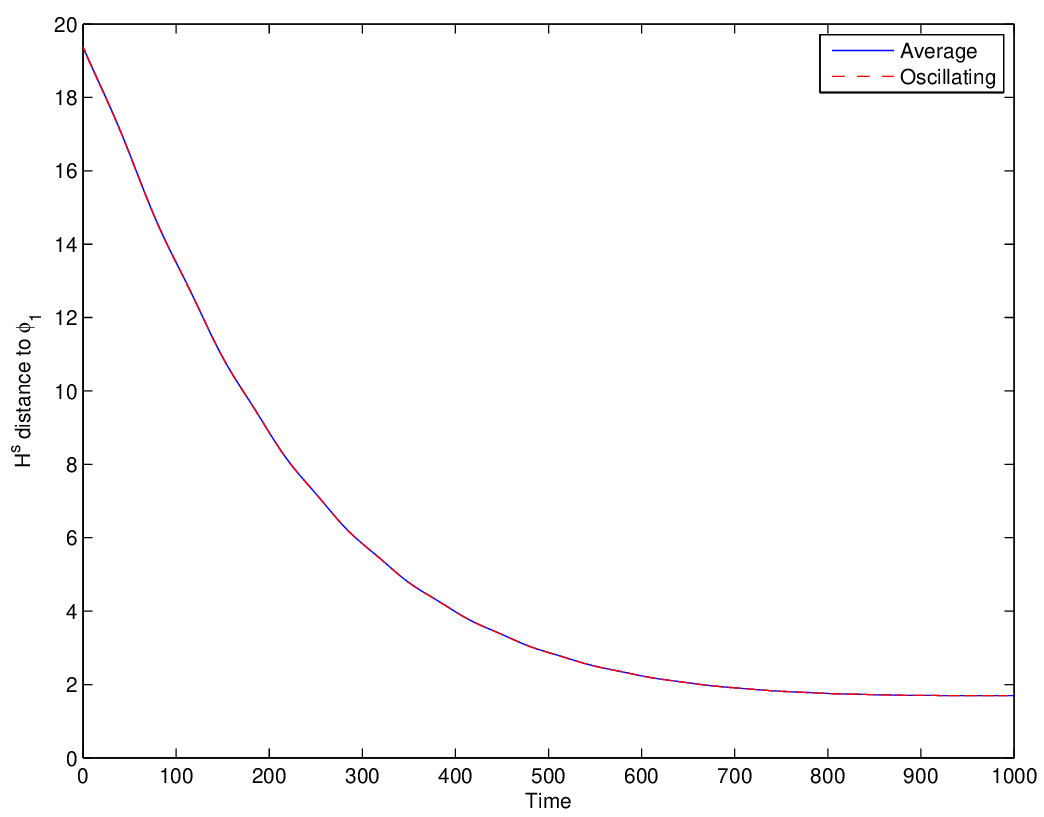}
\hspace{1cm}
\includegraphics[width=5cm]{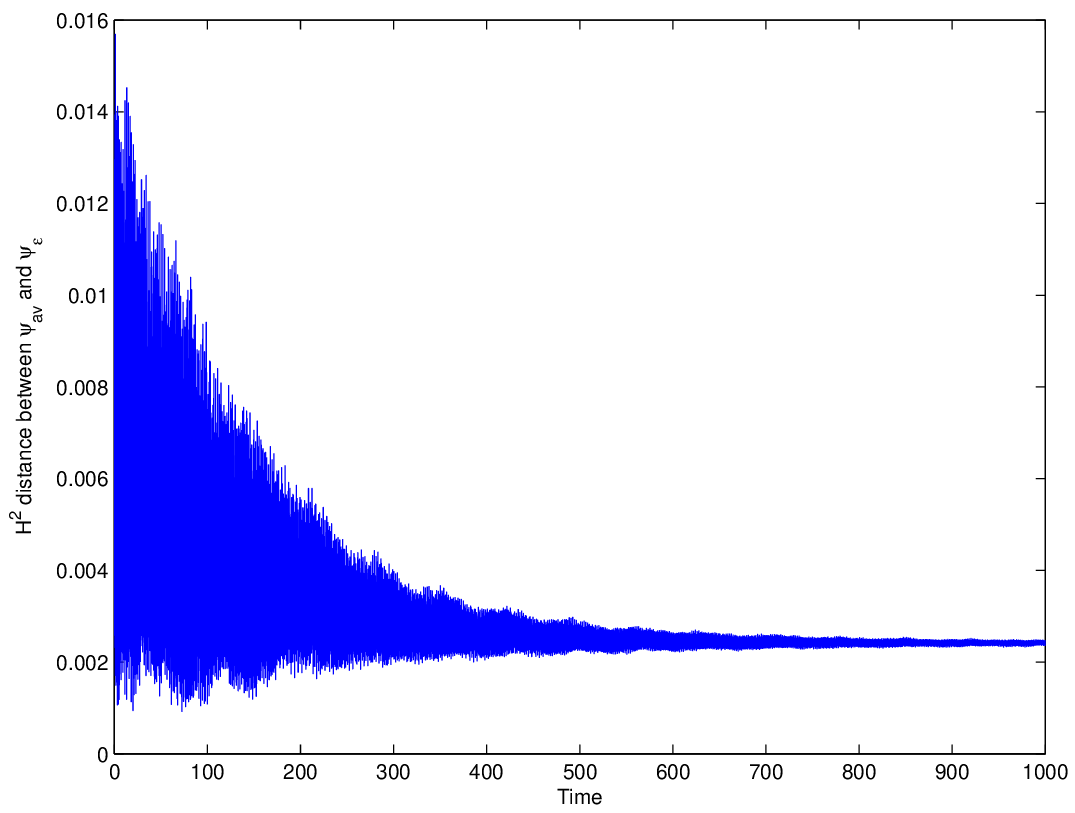}
\caption[]
{\label{fig_epsilon_3} $H^s$ norm to the ground state
(left) for the averaged system (continuous line) and the oscillating system (dashed line). $H^2$ gap from the average (right).}
\end{figure}
\begin{figure}[H]
\centering
\includegraphics[width=5cm]{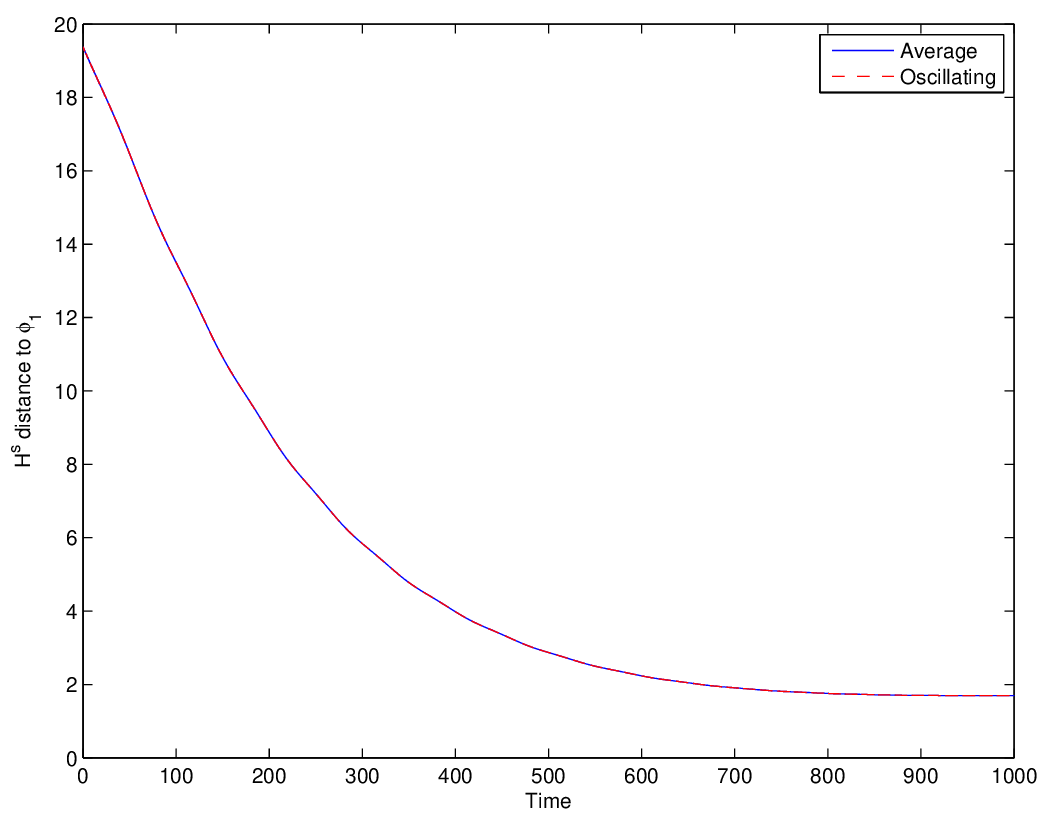}
\hspace{1cm}
\includegraphics[width=5cm]{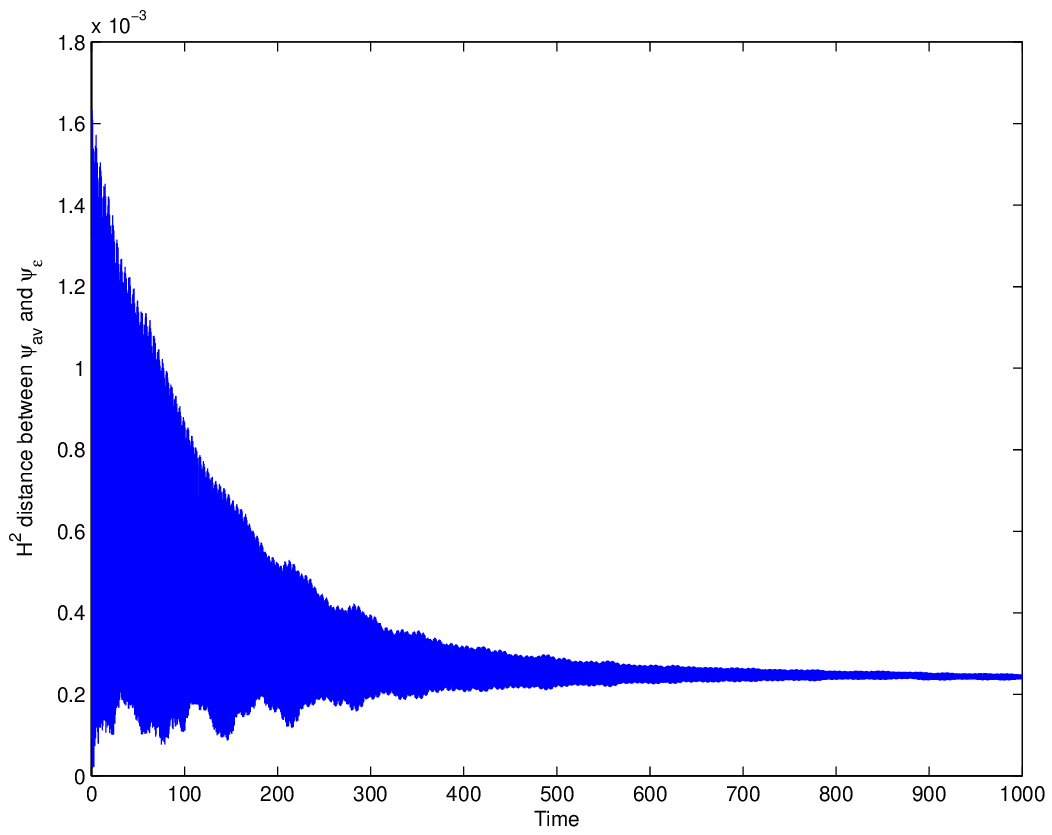}
\caption[]
{\label{fig_epsilon_4} $H^s$ norm to the ground state
(left) for the averaged system (continuous line) and the oscillating system (dashed line). $H^2$ gap from the average (right).}
\end{figure}

\begin{rmq}
Although the time scales at stake in these simulations can seem very large, one has to remember that the Schr\"odinger equation is considered in atomic unity.
\end{rmq}

\section{Conclusion, open problems and perspectives}

$\indent$ In this article we have defined explicit oscillating controls that drive the solution of our system arbitrarily close to the ground state provided that the control is oscillating enough and the time is large enough. To achieve this we have used and developed tools from the theory of finite dimension dynamical systems and applied them to the considered Schr\"odinger equation. We managed by adding a mathematically and physically meaningful term to weaken the previous assumptions on the coupling realized by this model. The assumptions that were made are proved to be generic with respect to the functions determining the system (potential, dipolar and polarizability moments). 
The results presented should be generalizable to a compact manifold with the Laplace-Beltrami operator. We performed numerical simulations to illustrate the approximate controllability. This gives numerical bounds on the time scale and on the values of the oscillating parameter needed to drive any initial condition arbitrarily close to the ground state.

A challenging question would be to prove an approximation property of the averaged system on an infinite time interval $[s,+\infty)$. This would lead to approximate stabilization to the ground state. Based on the numerical simulations, this result seems to hold. Unfortunately the tools developed here are really based on the finite time interval and cannot be extended directly. 
\noindent
In \cite{BeauchardLaurent}, Beauchard and Laurent proved the local exact controllability in $H^3$ around the ground state for the system (\ref{syst}) in the dipolar approximation (i.e. $Q_2 \equiv 0$) under some coupling assumptions in one dimension. If one manages to extend their result to the system (\ref{syst}) with suitable assumptions on $Q_2$, this may lead to a global exact controllability result around the ground state, at least for the one dimensional case. The main difficulty would be to obtain the approximate convergence in the same functional setting as their local exact controllability result and with coherent assumptions on the polarizability moment.


\paragraph*{Acknowledgements :} The author would like to thank K. Beauchard for having interested him in this problem and for fruitful discussions.


\nocite{MRT}

\bibliography{biblio}

\newcommand{\etalchar}[1]{$^{#1}$}
\begin{thebibliography}{CdANB}

\bibitem[B]{Beauchard05}
K.~Beauchard.
\newblock Local controllability of a 1-{D} {S}chr\"odinger equation.
\newblock {\em J. Math. Pures Appl. (9)}, 84(7):851--956, 2005.

\bibitem[BCC]{BoussaidCaponigroChambrion}
N.~Boussaid, M.~Caponigro, and T.~Chambrion.
\newblock {Weakly-coupled systems in quantum control}.
\newblock INRIA Nancy-Grand Est ''CUPIDSE'' Color program., September 2011.

\bibitem[BCCS]{BCCS11}
U.~Boscain, M.~Caponigro, T.~Chambrion, and M.~Sigalotti.
\newblock {A weak spectral condition for the controllability of the bilinear
  Schr{\"o}dinger equation with application to the control of a rotating planar
  molecule}.
\newblock October 2011.

\bibitem[BL]{BeauchardLaurent}
K.~Beauchard and C.~Laurent.
\newblock Local controllability of 1{D} linear and nonlinear {S}chr\"odinger
  equations with bilinear control.
\newblock {\em J. Math. Pures Appl. (9)}, 94(5):520--554, 2010.

\bibitem[BM]{BeauchardMirrahimi09}
K.~Beauchard and M.~Mirrahimi.
\newblock Practical stabilization of a quantum particle in a one-dimensional
  infinite square potential well.
\newblock {\em SIAM J. Control Optim.}, 48(2):1179--1205, 2009.

\bibitem[BMS]{BallMarsdenSlemrod82}
J.~M. Ball, J.E. Marsden, and M.~Slemrod.
\newblock Controllability for distributed bilinear systems.
\newblock {\em SIAM J. Control Optim.}, 20(4):575--597, 1982.

\bibitem[BN]{BeauchardNersesyan}
K.~Beauchard and V.~Nersesyan.
\newblock Semi-global weak stabilization of bilinear {S}chr\"odinger equations.
\newblock {\em C. R. Math. Acad. Sci. Paris}, 348(19-20):1073--1078, 2010.

\bibitem[Ca]{Caz}
T.~Cazenave.
\newblock {\em Semilinear {S}chr\"odinger equations}, volume~10 of {\em Courant
  Lecture Notes in Mathematics}.
\newblock New York University Courant Institute of Mathematical Sciences, New
  York, 2003.

\bibitem[Cor]{CoronBook}
J.-M. Coron.
\newblock {\em Control and nonlinearity}, volume 136 of {\em Mathematical
  Surveys and Monographs}.
\newblock American Mathematical Society, Providence, RI, 2007.

\bibitem[Cou1]{Couchouron1}
J.-F. Couchouron.
\newblock Compactness theorems for abstract evolution problems.
\newblock {\em J. Evol. Equ.}, 2(2):151--175, 2002.

\bibitem[Cou2]{Couchouron2}
J.-F. Couchouron.
\newblock Strong stabilization of controlled vibrating systems.
\newblock {\em ESAIM : COCV}, 2010.
\newblock DOI : 10.1051/cocv/2010041.

\bibitem[CdAN]{CAN}
J.-M. Coron and B.~d'~Andréa-Novel.
\newblock Stabilization of a rotating body beam without damping.
\newblock {\em IEEE Trans. Automat. Control}, 43(5):608--618, 1998.

\bibitem[CdANB]{CoronAndreaNovelBastin07}
J.-M. Coron, B.~d'~Andréa-Novel, and G.~Bastin.
\newblock A strict {L}yapunov function for boundary control of hyperbolic
  systems of conservation laws.
\newblock {\em IEEE Trans. Automat. Control}, 52(1):2--11, 2007.

\bibitem[CGLT]{CGLT}
J.-M. Coron, A.~Grigoriu, C.~Lefter, and G.~Turinici.
\newblock Quantum control design by {L}yapunov trajectory tracking for dipole
  and polarizability coupling.
\newblock {\em New. J. Phys.}, 11(10), 2009.

\bibitem[CMSB]{CMSB09}
T.~Chambrion, P.~Mason, M.~Sigalotti, and U.~Boscain.
\newblock Controllability of the discrete-spectrum {S}chr\"odinger equation
  driven by an external field.
\newblock {\em Ann. Inst. H. Poincar\'e Anal. Non Lin\'eaire}, 26(1):329--349,
  2009.

\bibitem[DBA{\etalchar{+}}]{Dion_1}
C.M. Dion, A.D. Bandrauk, O.~Atabek, A.~Keller, H.~Umeda, and Y.~Fujimura.
\newblock Two-frequency {I}{R} laser orientation of polar molecules. numerical
  simulations for hcn.
\newblock {\em Chem. Phys. Lett.}, (302):215--223, 1999.

\bibitem[DKAB]{Dion_2}
C.M. Dion, A.~Keller, O.~Atabek, and A.D. Bandrauk.
\newblock Laser-induced alignment dynamics of {HCN} : Roles of the permanent
  dipole moment and the polarizability.
\newblock {\em Phys. Rev.}, (59):1382, 1999.

\bibitem[EP]{ErvedozaPuel09}
S.~Ervedoza and J.-P. Puel.
\newblock Approximate controllability for a system of {S}chr\"odinger equations
  modelling a single trapped ion.
\newblock {\em Ann. Inst. H. Poincar\'e Anal. Non Lin\'eaire},
  26(6):2111--2136, 2009.

\bibitem[GLT]{GrigoriuLefterTurinici09}
A.~Grigoriu, C.~Lefter, and G.~Turinici.
\newblock Lyapunov control of {S}chr{\"o}dinger equation:beyond the dipole
  approximations.
\newblock In {\em {Proc of the 28th IASTED International Conference on
  Modelling, Identification and Control}}, pages 119--123, Innsbruck, Austria,
  2009.

\bibitem[HL]{HaleLunel90}
J.K. Hale and S.M. Lunel.
\newblock Averaging in infinite dimensions.
\newblock {\em J. Integral Equations Appl.}, 2(4):463--494, 1990.

\bibitem[M]{Mirrahimi09}
M.~Mirrahimi.
\newblock Lyapunov control of a quantum particle in a decaying potential.
\newblock {\em Ann. Inst. H. Poincar\'e Anal. Non Lin\'eaire},
  26(5):1743--1765, 2009.

\bibitem[MRT]{MRT}
M.~Mirrahimi, P.~Rouchon, and G.~Turinici.
\newblock Lyapunov control of bilinear {S}chr\"odinger equations.
\newblock {\em Automatica J. IFAC}, 41(11):1987--1994, 2005.

\bibitem[MSR]{MirrahimiSarletteRouchon10}
M.~Mirrahimi, A.~Sarlette, and P.~Rouchon.
\newblock Real-time synchronization feedbacks for single-atom frequency
  standards: V- and lambda-structure systems.
\newblock In {\em Proceedings of the 49th IEEE Conference on Decision and
  Control}, pages 5031--5036, Atlanta, December 2010.

\bibitem[N1]{Nersesyan}
V.~Nersesyan.
\newblock Growth of {S}obolev norms and controllability of the {S}chr\"odinger
  equation.
\newblock {\em Comm. Math. Phys.}, 290(1):371--387, 2009.

\bibitem[N2]{Nersesyan10}
V.~Nersesyan.
\newblock Global approximate controllability for {S}chr\"odinger equation in
  higher {S}obolev norms and applications.
\newblock {\em Ann. Inst. H. Poincar\'e Anal. Non Lin\'eaire}, 27(3):901--915,
  2010.

\bibitem[R]{RouchonModele}
P.~Rouchon.
\newblock Control of a quantum particle in a moving potential well.
\newblock In {\em Lagrangian and {H}amiltonian methods for nonlinear control
  2003}, pages 287--290. IFAC, Laxenburg, 2003.

\bibitem[SVM]{SandersVerhulst07}
J.~A. Sanders, F.~Verhulst, and J.~Murdock.
\newblock {\em Averaging methods in nonlinear dynamical systems}, volume~59 of
  {\em Applied Mathematical Sciences}.
\newblock Springer, New York, second edition, 2007.

\bibitem[T]{Turinici07}
G.~Turinici.
\newblock Beyond bilinear controllability: applications to quantum control.
\newblock In {\em Control of coupled partial differential equations}, volume
  155 of {\em Internat. Ser. Numer. Math.}, pages 293--309. Birkh\"auser,
  Basel, 2007.

\end{thebibliography}
\bibliographystyle{mcss} 

\end{document}